\documentclass[11pt,a4paper,leqno]{amsart}
\usepackage{amsfonts}
\usepackage{amsthm}
\usepackage{amsmath}
\usepackage{amscd}
\usepackage[latin2]{inputenc}
\usepackage{t1enc}
\usepackage[mathscr]{eucal}
\usepackage{indentfirst}
\usepackage{graphicx}
\usepackage{graphics}
\usepackage{pict2e}
\usepackage{epic}
\numberwithin{equation}{section}
\usepackage[margin=2.9cm]{geometry}
\usepackage{epstopdf} 
\RequirePackage{doi}
\usepackage{hyperref}
\allowdisplaybreaks
\usepackage{cite}

\usepackage{verbatim,amssymb,amsfonts}
\usepackage{mathrsfs}
\usepackage{mathtools}
\usepackage{abraces}
\usepackage{tikz-cd}
\usepackage{indentfirst}
\usepackage{slashed}
\usepackage{thmtools}
\usepackage{setspace}
\usepackage{enumerate}

\theoremstyle{plain}
\newtheorem{theorem}{Theorem}[section]
\newtheorem{lemma}[theorem]{Lemma}
\newtheorem{corollary}[theorem]{Corollary}
\newtheorem{proposition}[theorem]{Proposition}

\newtheorem{fact}{Fact}

\newtheorem{assumption}{Assumption}

 \theoremstyle{definition}
\newtheorem{definition}[theorem]{Definition}
\newtheorem{remark}[theorem]{Remark}

\newcommand{\im}{\operatorname{im}}

\DeclarePairedDelimiterX{\inp}[2]{\langle}{\rangle}{#1, #2}

\newcommand{\cE}{{\mathcal E}}

\newcommand{\cV}{{\mathcal V }}

\newcommand{\id}{{\text{Id} }}

\newcommand{\Td}{{\text{Td} }}

\newcommand{\D}{{\slashed{\partial} }}
\newcommand{\spinc}{{ \text{Spin}^c }}
\newcommand{\diag}{{ \text{Diag} }}
\newcommand{\ind}{{ \text{Ind} }}
\newcommand{\proj}{{ \text{proj} }}
\newcommand{\tr}{{ \text{Tr} }}
\newcommand{\pt}{{ \text{pt} }}

\newcommand{\ba}{\begin{eqnarray}}
\newcommand{\na}{\end{eqnarray}}
\newcommand{\ban}{\begin{eqnarray*}}
\newcommand{\nan}{\end{eqnarray*}}

\newcommand{\C}{{\mathbb C}}

\newcommand{\Q}{{\mathbb Q}}
\newcommand{\R}{{\mathbb R}}
\newcommand{\Z}{{\mathbb Z}}

\begin{document}

\title{Analytic Pontryagin Duality}

\author[Johnny Lim]{Johnny Lim}

\address{Johnny Lim \\ Department of Pure Mathematics \\
University of Adelaide 5005, Australia.} 

\email{johnny.lim@adelaide.edu.au}




\begin{abstract}

Let $X$ be a smooth compact manifold. We propose a geometric model for the group $K^0(X,\R/\Z).$ We study a well-defined and non-degenerate analytic duality pairing between $K^0(X,\R/\Z)$ and its Pontryagin dual group, the Baum-Douglas geometric $K$-homology $K_0(X),$ whose pairing formula  comprises of an analytic term involving the Dai-Zhang eta-invariant associated to a twisted Dirac-type operator and a topological term involving a differential form and some characteristic forms. 
This yields a robust $\R/\Z$-valued invariant. We also study two special cases of the analytic pairing of this form in the cohomology  group $H^1(X,\R/\Z)$ and  $H^2(X,\R/\Z).$ 
\end{abstract}

\maketitle

\vspace{-1em}
\section{Introduction} 

The purpose of this paper is to introduce an $\R/\Z$-valued invariant defined by an analytic duality pairing between the even $K$-theory with coefficients in $\R/\Z$ and the even Baum-Douglas geometric $K$-homology \cite{baum1982k},  
$$K^0(X,\R/\Z) \times K_0(X,\Z) \to \R/\Z,$$ for a smooth compact manifold $X.$ It is commonly known as the \textit{Pontryagin duality} pairing. By the Universal Coefficient Theorem in $K$-theory \cite{yosimura1975universal}, there is a short exact sequence 
$$0 \to \text{Ext}(K_{-1}(X),\R/\Z) \to K^0(X,\R/\Z) \to \text{Hom}(K_0(X),\R/\Z) \to 0.$$ Since $\R/\Z$ is divisible, the vanishing of the $\text{Ext}$ group implies a natural isomorphism 
$K^0(X,\R/\Z) \xrightarrow{\sim} \text{Hom}(K_0(X),\R/\Z).$ We formulate an \textit{analytic} pairing implementing the isomorphism.  By `analytic', we mean that the pairing involves the \text{eta-invariant} associated to a Dirac-type operator twisted by some pullback bundle over a smooth compact manifold.

This is inspired by the work of Lott \cite{lott1994r} on the $\R/\Z$ index theory. As motivated by Karoubi's model  of $K$-theory with coefficients \cite{karoubi2008k} and the index theorem for flat bundles of Atiyah-Patodi-Singer \cite{atiyah1976spectral3}, Lott  formulated an analytic $K^1$-pairing $K^1(X,\R/\Z) \times K_1(X)$ in terms of the eta-invariant of Atiyah-Patodi-Singer  \cite{atiyah1975spectral1}. 
In the physical aspect, such a pairing  has been observed by Maldacena-Seiberg-Moore  \cite{maldacena2001d} as describing the Aharonov-Bohm effect of $D$-branes in Type IIA String theory. An extended discussion of such a manifestation in String theory was given by Warren \cite{warren12}. Beyond this, there are several studies related to the $\R/\Z$ $K$-theory from different points of view. For instance, Basu \cite{basu2005k} provided a model via bundles of von Neumann algebras, according to the suggestion in \cite[Section 5, Remark 4]{atiyah1976spectral3}; Antonini \textit{et al} \cite{antonini2014flat} gave a construction of $\R/\Z$ $K$-theory via an operator algebraic approach; on the other hand, the strategy of Deeley \cite{deeley2012r} is rather different in that he studied the pairing between the usual $K$-theory and the $K$-homology with $\R/\Z$ coefficients.   

 However, there is no known work,  to the author's knowledge, on the \textit{direct} analog of the analytic $K^1$-pairing of Lott in the \textit{even} case. 
 This paper is aimed to fill in this gap. We construct a geometric model of the group $K^0(X,\R/\Z),$  whose cocycle is a triple consisting of an element $g$ of $K^1(X),$ a pair of flat connections $(d,g^{-1}dg)$ on a trivial bundle and an even degree differential form $\mu$ on $X$ satisfying  a certain exactness condition on the odd Chern character of $g.$ 
Its pairing with an even geometric $K$-cycle $(M,E,f)$  can then  be explicitly described by an (reduced) even eta-type invariant of some twisted Dirac-type operator on a cylinder $M\times [0,1],$  which appears as one of the boundary correction terms in the Dai-Zhang Toeplitz index theorem on manifolds with boundary \cite{dai2006index},  and a topological term, whose integrand is the wedge product of the pullback of $\mu$ and some characteristic forms on $M.$ The  resulting $\R/\Z$-valued invariant is robust in the sense that it is independent of the  geometry of the underlying manifold and the bundle. 
We also show that such an analytic pairing  is  non-degenerate, and thus it is a valid implementation of the isomorphism $K^0(X,\R/\Z) \xrightarrow{\sim} \text{Hom}(K_0(X),\R/\Z).$ As an intermediate step, we consider the special case of $n$-spheres. This provides a non-trivial example of the pairing.  In terms of application, we believe that the analytic pairing in $K^0$ describes the Aharonov-Bohm effect of  $D$-branes in Type IIB String theory. 

As a motivation, we begin by studying two non-trivial special cases of the analytic pairing in the $\R/\Z$-cohomology of degree one and two. In the case of $H^2,$ by the pullback via a smooth map, we investigate the pairing on  $H^2(S^2,\R/\Z),$ whose representative is a pure Hermitian local line bundle introduced by Melrose \cite{melrose2004star}. A local line bundle is \textit{projective} in that it is defined locally over a neighbourhood of the diagonal. Thus, the corresponding twisted Dirac operator is projective \textit{ala} Mathai, Melrose and Singer \cite{mathai2006fractional,mathai2005index}. These are projective differential operators with kernels whose supports  are contained in the diagonal of $S^2.$ 
The caveat is that these operators do not have a spectrum and thus do not have a well-defined eta-invariant. We make several assumptions and define a variant of the Dai-Zhang eta-invariant for twisted projective Dirac operators in the special case of $S^2.$ 
On the other hand, the analytic pairing in $H^1$ is less complicated. The pairing consists of the Atiyah-Patodi-Singer eta-invariant of the Dirac operator on $S^1$ twisted by a flat bundle and the holonomy of a flat connection over  $S^1.$ This can be viewed as a special case of the analytic $K^1$-pairing.

This paper is organised as follows. We discuss the analytic $H^2$-pairing in Section 2 and the analytic $H^1$-pairing in Section 3.  In Section 4, we propose a model for the group $K^0(X,\R/\Z)$ and study its properties in detail. 
Then, we  state and explain our main result in Section 5 and devote the whole of Section 6 to its full proof.  \newline

\textbf{Acknowledgements.} The author acknowledges support (Divisional Scholarships) from the Faculty of Engineering, Computer \& Mathematical Sciences (ECMS) of the University of Adelaide. He is grateful to his PhD supervisor Elder Professor Varghese Mathai for suggesting the problems in the paper and for his helpful guidance. He would like to thank Dr Hang Wang, Dr Peter Hochs, Dr Guo Chuan Thiang and others for their helpful advice. The author also thank David Brook for proofreading the manuscript. 

A student talk by the author based on this paper, received the \textbf{Elsevier Young Scientist Award} for the best student talk at the Chern Institute of Mathematics conference entitled ``Index Theory, Duality and Related Fields'', June 17-21, 2019 at Tianjin, China.

\vspace{-0.5em}
\section{Analytic duality pairing   $H^2(X,\R/\Z) \times H_2(X,\Z)$ }
In this section, we study  the analytic Pontryagin duality pairing in the $\R/\Z$-cohomology of degree two.
Let $X$ be a smooth compact manifold.  The classical (topological)  pairing 
\begin{equation}\label{eq:topo0}
H^2(X,\R/\Z) \times H_2(X,\Z) \to \R/\Z
\end{equation}
 is given  by the holonomy of the pullback of a representative $\omega$ in $H^2(X,\R/\Z)$ over some singular cycle $c$ in $H_2(X,\Z)$  via a continuous map $f: c \to X.$  
\begin{fact}[{\cite[\text{IV. 7.35}]{rudjak1998thom}}] \label{homobordism1}
	Every homology class $z \in H_i(X,\Z)$ with $i \leq 6$ can be  represented by a smooth manifold. Let $\Omega^{or}_i(X)$ be the $i$-th oriented bordism group of $X.$ The map $$ \Omega^{or}_i(X) \to H_i(X,\Z); \;\;\;\;(S,f) \mapsto f_*[S]$$ is an isomorphism for $i \leq 3.$
\end{fact} Without loss of generality, we replace $c$ by an oriented, connected, closed Riemannian surface $\Sigma.$ Let $[f: \Sigma \to X] \in H_2(X),$ with the equivalence relation given by thin bordism, cf. \cite{rudjak1998thom}.   Then, the  pairing \eqref{eq:topo0} can be expressed as 
\begin{equation}\label{eq:topo1}
\big( \omega, [\Sigma \xrightarrow{f} X] \big) \mapsto \int_\Sigma f^*\omega \text{ mod }\Z.
 \end{equation}
By classification results, there are 3 cases: $\Sigma_0=S^2$ (of genus zero), $\Sigma_1=T^2=S^1 \times S^1$ (of genus 1), and in general $\Sigma_{2g}=T^{\# g}$ (of genus $2g$ for $g>1).$ Since genus($T^{\# g})>$ genus($S^2$), there exists a degree 1 map $\phi: T^{\# g} \to S^2,$ see \cite{hatcher2005algebraic}. Hence, it suffices to consider  $\Sigma = S^2,$ and the other cases follow by the composition 
\begin{center}
\begin{tikzcd}
	 \Sigma_{2g} \arrow[r,"f'"] & X. \\
	 S^2 \arrow[u,"\phi"] \arrow[ur,"f"] &
\end{tikzcd}
\end{center}
From \eqref{eq:topo1}, this reduces to the analytic pairing on $S^2$ by pullback.  In the literature, the geometric object associated to $f^*\omega$ is often known as a \textit{flat gerbe with connection} over $S^2,$ cf. \cite{hitchin2001lectures}. However, it is not clear how to `twist' a Dirac operator on $S^2$ with a gerbe. To circumvent this, we use the Hermitian local bundles of Melrose \cite{melrose2004star}.

\subsection{Representative of $H^2(S^2,\R/\Z)$  as projective line bundles}

For $i=2,3,4,$ let $\diag_i=\{(x,\ldots,x) \in S^2 \times \cdots \times S^2\}$ be the diagonal of  $S^2.$

 \begin{definition}[\cite{melrose2004star}]
 	A Hermitian local line bundle $L$  over $S^2$ is a complex line bundle over a neighbourhood $V_2$ of the diagonal  $\diag_2,$ together with a  tensor product isomorphism of smooth bundles
 	\begin{equation}\label{eq:compo0}
 	\pi^*_3 L \otimes \pi^*_1 L \xrightarrow{\cong} \pi^*_2 L
 	\end{equation} over a neighbourhood $V_3$  of the diagonal $\diag_3,$ where $\pi_i : S^2 \times S^2 \times S^2 \to S^2 \times S^2$ is the projection omitting the $i$-th element $\pi_i (x_1,x_2,x_3)=(\widehat{x_i}),$ and satisfying the associativity condition $L_{(x,y)} \otimes L_{(y,z)}  \otimes L_{(z,t)}  \to L_{(x,t)} $ on a sufficiently small neighbourhood of the diagonal $\diag_4.$ 
 \end{definition}

Strictly speaking, $L$ is \textit{not} a genuine line bundle but is only projective in the sense of \cite{mathai2005index}. It is only defined locally over some neighbourhood of the diagonal. More precisely, choose a good cover $\{U_i\}$ of $S^2,$ then the product $U_i \times U_i$ defines an open cover of $\diag_2,$ which is contained in small neighbourhood $V_2,$ i.e. $\diag_2 \subset U_i \times U_i \subset V_2 \subset S^2 \times S^2.$ Choose $p_i \in U_i$ and consider the `left' and `right' bundles 
\begin{equation}
\mathcal{L}_{i,p_i} = L |_{U_i \times \{p_i\}}, \;\; \mathcal{R}_{p_i,i}= L|_{\{p_i \} \times U_i}.
\end{equation}
Then, by the composition law \eqref{eq:compo0},  a line bundle 
\begin{equation} \label{eq:identification0}
L = \mathcal{L}_{i,p_i}  \otimes \mathcal{R}_{p_i,i} 
\end{equation} is defined over $U_i \times U_i.$ Moreover, there is a dual bundle identification $\mathcal{R}_{p_i,i} \cong \mathcal{L}^{-1}_{i,p_i}$ over $U_i.$  By \cite[Lemma 1]{melrose2004star}, a local line bundle $L$ on $S^2$ can be equipped with a multiplicative unitary structure and a multiplicative Hermitian connection. A connection $\nabla$ is multiplicative if for a local section $u$ of $L$ near $(x,y) \in U_i \times U_i$ with $\nabla u=0$ at $(x,y),$ and for a local section $v$ of $L$ near $(y,z) \in U_i \times U_i$ with $\nabla v=0$ at $(y,z),$   the composition $C(u,v)$ of \eqref{eq:compo0} is locally constant at $(x,y,z).$  

The multiplicative Hermitian structure is taken as a Hermitian structure $g(\cdot,\cdot)_i$ on each $\mathcal{L}_{i,p_i}.$ Using the dual identification on $\mathcal{R}_{p_i,i},$ this defines a Hermitian structure on it over $U_i$ and thus $L$ over $U_i \times U_i$ via \eqref{eq:identification0}. Using a partition of unity $\rho_i$ subordinate to $U_i\times U_i,$ the Hermitian structure $g(u,v)= \sum_i (\rho_i \times \rho_i) g(u,v)_i$ is well-defined. 

By \cite[Proposition 2]{melrose2004star}, there is a one-to-one correspondence between the group $H^2(S^2,\R)$ and the set of  Hermitian local line bundles modulo unitary multiplicative isomorphisms in some neighbourhood of the diagonal. In particular, the representative closed 2-forms are identified with the curvature of the Hermitian local line bundle, i.e. $[B/2\pi] \in H^2(S^2,\R)$ and $B=\nabla\circ \nabla$ for $(L,\nabla).$ It is the first Chern class of $L.$ In this way, we have obtained another geometric interpretation of $H^2(S^2,\R/\Z).$ 

\begin{lemma}
	The group $H^2(S^2,\R/\Z)$ is isomorphic to the quotient of $H^2(S^2,\R)$ by the reduced cohomology $\tilde{H}^2(S^2,\Z).$
\end{lemma}

\begin{proof}
	Consider the long exact sequence 
	\begin{equation}\label{eq:les0}
	\cdots  \to  H^1(S^2,\R/\Z) \xrightarrow{c_1} H^2(S^2,\Z) \to H^2(S^2,\R) \to H^2(S^2,\R/\Z) \to H^3(S^2,\Z) \to  \cdots 
	\end{equation} where the first\footnote{For clarity, the notation $H^1(S^1,\R/\Z)$ denotes the circle group $\R/\Z\cong U(1)$ equipped with the discrete topology. This should not be confused with the notation $H^1(S^2, \underline{U(1)}) \cong H^2(S^2,\Z)$  where $\underline{U(1)}$ denotes the sheaf of germs of $U(1)$-valued functions on $X.$  In particular, by standard bundle theory the latter classifies all principal $U(1)$-bundles, in which $U(1)$ is the circle group equipped with the usual topology. } map $c_1 : H^1(S^2,\R/\Z) \to H^2(S^2,\Z)$ is the first Chern class. Let  $L_0$ be a flat line bundle over $S^2$. Then, there are two cases: $$c_1(L_0)=0\;\; \text{  or  } \;\; c_1(L_0) \in H^2_{tors}(S^2,\Z).$$ Since $H^2(S^2,\Z) \cong \Z$ is non-torsion, we have $c_1(L_0) \equiv 0.$ So, all such flat line bundles are necessarily trivial. They are labelled by the integer $0$ in $\Z.$ Let $\tilde{H}^2(S^2,\Z)$ be the group generated by the Bott bundle\footnote{Here $\tau$ denotes the tautological non-trivial line bundle over $S^2 \cong \C P^1.$  } $\beta=\tau -1$ which corresponds to the generator $1 \in \Z.$ Since $H^3(S^2,\Z)=0,$ from \eqref{eq:les0} we get
	\begin{equation}\label{eq:ses0}
	0 \to H^2(S^2,\Z)/\text{im}(c_1) \to H^2(S^2,\R) \to H^2(S^2,\R/\Z) \to 0
	\end{equation}
	and thus 
	$H^2(S^2,\R/\Z) \cong H^2(S^2,\R)/\tilde{H}^2(S^2,\Z).$
\end{proof} 
\begin{remark}
In other words, we interpret an element in $H^2(S^2,\R/\Z)$ as a \textit{pure} Hermitian local line bundle $L$ over $S^2,$ in the sense that it is `trivial'  when it descends to an ordinary non-trivial line bundle on $S^2.$    
\end{remark}

\subsection{Projective Dirac  operator on $S^2$ twisted by $L$}
Let $L$ be a pure Hermitian local line bundle with connection over $S^2$ defined above, whose (normalised) curvature is a representative in $H^2(S^2,\R).$  
The appropriate notion of the twisted Dirac operator is the \textit{projective Dirac operator} $\D^L_{S^2,\proj}$ introduced by Mathai, Melrose and Singer, cf. \cite{mathai2005index,mathai2006fractional,melrose2004star}. See also \cite{mathai2008equivariant} for its relation with transversally elliptic operators.  Such an operator is a projective elliptic differential operator of order one defined on some neighbourhood of the diagonal $\diag_2,$ with its kernel supported on the intersection of that neighbourhood and where $L$ exists. From \cite{mathai2006fractional},  there is a projective spinor bundle $S$ over $S^2$ associated to the Azumaya bundle $\C l(TS^2).$ 
Since $S^2$ is $\spinc,$ it can be viewed as the lift of the ordinary spinor bundle, also denoted as $S,$ trivially to some $\epsilon$-neighbourhood $N_\epsilon$ of the diagonal. Then, the projective bundle $S\otimes L$ is defined over $$N'_\epsilon := N_\epsilon \cap U_i \times U_i \supset \diag_2.$$ Let $\nabla^{S \otimes L}$ be the tensor product connection, defined by taking an appropriate partition of unity subordinate to $N'_\epsilon.$ 
Such a tensor product connection always exists, by the existence of the multiplicative Hermitian connection of $L$ defined above, and the restriction to $N'_\epsilon$ of a global spin connection on $S.$ 
The projective Dirac operator is given in terms of distributions
\begin{equation} \label{eq:projDirac0}
\D^L_{S^2,\proj} :=cl \cdot \nabla^{S \otimes L}_{\text{left}} (\kappa_\id)
\end{equation}
where $ \kappa_\id=\delta(z-z')\id_{S \otimes L}$ is the kernel of the identity operator in $\text{Diff}^1(S^2,S \otimes L);$  $\nabla^{S\otimes L}_{\text{left}}$ is the connection restricted to the left variables and  $cl$ denotes the Clifford action of $T^*S^2$ on the left.  The projective Dirac operator $ \D^L_{S^2,\proj}$ is elliptic and is odd with respect to the $\Z_2$-grading 
$$\D^{L,\pm}_{S^2,\proj} \in \text{Diff}^1(S^2;S^{\pm} \otimes L , S^{\mp}\otimes L).$$
By \cite[Theorem 1]{mathai2006fractional}, the projective analytical index of the positive part   $\D^{L,+}_{S^2,\proj}$  is given by 
\begin{equation}\label{eq:analprojind0}
\ind(\D^{L,+}_{S^2,\proj})= \tr(\D^{L,+}_{S^2,\proj} Q -1_{S^- \otimes L}) -\tr (Q\D^{L,+}_{S^2,\proj} -1_{S^+ \otimes L}) 
\end{equation} for \textit{any} parametrix $Q$ of $\D^{L,+}_{S^2,\proj}.$  
By \cite[Theorem 2]{melrose2004star} (and also \cite[Theorem 2]{mathai2006fractional}), the projective analog of the Atiyah-Singer index formula of the positive part is given by 
\begin{equation}\label{eq:projind1}
\ind\big(\D^{L,+}_{S^2,\proj}\big) = \int_{S^2} \Td(S^2) \wedge \exp(B/2\pi) \;\; \in \R
\end{equation} where $\exp(B/2\pi)$ denotes the first Chern class of the local line bundle $L.$ 

\subsection{Analytic pairing formula  in $H^2(S^2,\R/\Z)$  } To formulate the analytic pairing in  the case of  $H^2(S^2,\R/\Z),$ we need to consider the eta-invariant for projective Dirac operators. There are two subtleties here. Firstly, for parity reasons we need an even analog of the eta-invariant.  Secondly, the operator involved is projective and does not have a spectrum. Thus, there is no well-defined notion of spectral asymmetry yet. 

To tackle the first point, we adopt the Dai-Zhang eta-invariant \cite{dai2006index} of an elliptic  operator on $S^2.$ Up to this stage, the `even' eta-type invariant has not been defined. We will discuss these in Section 5. Hence, for the moment let us assume that there is such a spectral invariant for elliptic operators on $S^2.$ In this projective case, we have to make several assumptions.

\begin{definition}
Define 
\begin{equation}\label{eq:dzetaS2}
\eta_{DZ}\big(\D^L_{S^2,\proj} \big) := \eta_{APS}\big(\D^L_{S^2 \times S^1,\proj} \big)
\end{equation} where $\eta_{DZ}$ (resp.  $\eta_{APS}$) denote the (unreduced) eta-invariant of the projective Dirac operator on $S^2$ of Dai-Zhang (resp. on $S^2 \times S^1$ of Atiyah-Patodi-Singer).  
\end{definition} Both of the LHS and RHS of \eqref{eq:dzetaS2} are not well-defined, since these operators are projective. However, we can still work on the RHS. In particular, this definition is consistent with the construction of the Dai-Zhang eta-invariant, which is done on the extension of $S^2$ to the cylinder $S^2 \times [0,1].$ See \cite{dai2006index} or Section 5. Moreover, we use the fact that the projective analytical index $\ind(\D^{L,+}_{S^2,\proj})$ given by \eqref{eq:analprojind0} is non-zero. Then,  we circumvent the technical assumption in the Dai-Zhang construction (requiring vanishing index) by considering the gluing of the bundle data on both ends $S^2 \times \{0\}$ and $S^2 \times \{1\}$ by an $K^1$-element $g.$ Since $K^1(S^2)\equiv 0$ by Bott Periodicity, the bundle $S\otimes L,$ extended trivially over to $S^2 \times [0,1],$ is glued trivially without any twisting at either end. This justifies the notation $\D^L_{S^2 \times S^1,\proj}.$ 

Next, to calculate the RHS of \eqref{eq:dzetaS2} , we rewrite the operator  $\D^L_{S^2 \times S^1,\proj}$ as the \textit{sharp product} of elliptic operators on the product manifold
\begin{equation}\label{eq:sharpprod0}
R:= \D^L_{S^2 \times S^1,\proj}=\D^L_{S^2,\proj} \# \D_{S^1} = 
\begin{pmatrix} 
\D_{S^1} \otimes 1  & 1 \otimes \D^{L,-}_{S^2,\proj} \\ 
1 \otimes \D^{L,+}_{S^2,\proj}  &  -\D_{S^1} \otimes 1
\end{pmatrix}.
\end{equation} Here, $\D_{S^1}$ is the ordinary Dirac operator on $S^1$ given by $\D_{S^1}=-id/d\theta.$ Note that both of the usual Dirac operator $S^1$ and the projective Dirac operator $ \D^L_{S^2,\proj} $ are elliptic, and so is $\D^L_{S^2 \times S^1,\proj}.$ This can be seen from the square of \eqref{eq:sharpprod0} 
\begin{equation}
R^2=
\begin{pmatrix} 
\D_{S^1}^2 \otimes 1     +  1 \otimes  \D^{L,-}_{S^2,\proj}   \D^{L,+}_{S^2,\proj}  & 0 \\ 
0 &  \D_{S^1}^2 \otimes 1     +  1 \otimes  \D^{L,+}_{S^2,\proj} \D^{L,-}_{S^2,\proj} 
\end{pmatrix}.
\end{equation} 

Moreover, it is readily verified that $R$ is self-adjoint. Nevertheless, $R$ is still projective and does not have a spectrum. To interpret  the RHS term of \eqref{eq:dzetaS2}, we define a projective analog of the usual relation of the eta-invariant of the sharp product  \cite{atiyah1968index,gilkey2018invariance} .

\begin{definition} \label{etasharp}
	Let $P=P^{\pm}$  be a projective Dirac operator (with $P^+=(P^-)^*$) on an even dimensional closed manifold $M_1$ and let $A$ be an ordinary self-adjoint Dirac operator on an odd dimensional closed manifold $M_2.$ Let $R'$ be the sharp product of $P$ and $A,$ as an elliptic differential operator on the product manifold $M_1 \times M_2,$ given by the following  formula similar to \eqref{eq:sharpprod0}
	\begin{equation}\label{eq:sharpprod1}
	R':= P \# A = 
	\begin{pmatrix} 
	A \otimes 1  & 1 \otimes P^- \\ 
	1 \otimes P^+  &  -A \otimes 1
	\end{pmatrix}.
	\end{equation}
	Define its projective Atiyah-Patodi-Singer eta-invariant as 
	\begin{equation}\label{eq:etasharp0}
	\eta_{APS}(R') := \ind(P^+)\cdot \eta_{APS}(A)
	\end{equation}
	where $\ind(P^+)$ is the projective analytic index given by the similar formula \eqref{eq:analprojind0} and $\eta_{APS}(A)$ denotes the usual eta-invariant of $A.$
\end{definition}

\begin{remark}
 Note that it might be misleading to write $\eta_{APS}(R'),$ since $R'$ has no spectrum. The point here is that we view the LHS of \eqref{eq:etasharp0} as the projective analog of the measure of the spectral asymmetry, given by the product of the two terms on the RHS of \eqref{eq:etasharp0}. This is valid because  the projective analytical index is independent of the choice of parametrix $Q$ of $P^+$ and  the other term is just the usual Atiyah-Patodi-Singer eta-invariant.  
\end{remark}

\begin{remark}
Moreover, Definition 2.5 holds for the ordinary case: when both $P$ and $A$ are ordinary Dirac operators, or more generally elliptic differential \footnote{ See \cite[Pg 85]{atiyah1976spectral3} for this statement on the eta-invariant of the sharp product of two elliptic \textit{differential} operators. This is not true if either one is \textit{pseudodifferential}. One cannot apply the approximating-$R'$-by-pseudodifferential-operator argument under the natural Fredholm topology (cf. \cite[Sec. 3.7]{gilkey2018invariance}) since it is not clear that the eta-invariant is continuous in the Fredholm topology. However, by some perturbation method, Gilkey \cite[Sec 3.8.4]{gilkey2018invariance} shows that it still holds when $P$ or $A$ is pseudodifferential.} operators.  For the benefit of the reader, we illustrate an argument in \cite{gilkey2018invariance}  on the equality of \eqref{eq:etasharp0} when $P$ and $A$ and thus $R'$ are elliptic differential. Let $\Delta^+=P^*P$  and $\Delta^-=PP^*$ be the associated Laplacians. Let $\{\lambda_i, \nu_i \}$ be a spectral resolution of $\Delta^+$ on
\begin{equation}
\ker(\Delta^+)^\perp=\text{Range}(P^-). 
\end{equation}
Then, $\{\lambda_i, P \nu_i/\sqrt{\lambda_i}\}$ is a spectral resolution of $\Delta^-$ on 
\begin{equation}
\ker(\Delta^-)^\perp=\text{Range}(P^+). 
\end{equation}
Observe that on the space $V_i= \text{span}\{(\nu_i \oplus 0), (0 \oplus P\nu_i /\sqrt{\lambda_i})\},$ the operator $R'$ is given by 
\begin{equation}\label{eq:sharpprod2}
R'_i=
\begin{pmatrix} 
k  &  \sqrt{\lambda_i}  \\ 
\sqrt{\lambda_i}  &  -k 
\end{pmatrix},
\end{equation} which has eigenvalues $$\pm \sqrt{k^2 + \lambda_i}.$$ Since $\lambda_i>0,$ the eigenvalues are non-zero and taking the eta-invariant is equivalent to taking the summation of these eigenvalues, which is zero. So, on $V_i$ it does not contribute to the eta. On the other hand, the complement of $\oplus_i V_i$ is 
\begin{equation}
W=(\ker(\Delta^+) \oplus 0) \oplus (0 \oplus \ker(\Delta^-)).
\end{equation} 
On $W,$ the operator $R'$ is given by 
\begin{equation}\label{eq:sharpprod3}
R'=
\begin{pmatrix} 
k\cdot \pi_{\ker(\Delta^+)}  &  0  \\ 
0  &  -k\cdot \pi_{\ker(\Delta^-)} 
\end{pmatrix}.
\end{equation}  
Then,  taking the eta is equivalent to taking the normalised trace of \eqref{eq:sharpprod3}, which gives
\begin{equation}\label{eq:sharpeta0}
\eta(R')= \sum \text{sgn}(k) \cdot [\tr(\pi_{\ker(\Delta^+)})-\tr(\pi_{\ker(\Delta^-)})] = \eta(A) \cdot \ind(P^+).
\end{equation}
Unfortunately, this does not extend to the projective case. In particular, the equality of $PQ-1=\pi_{\ker(\Delta^{+})}$ and $QP-1=\pi_{\ker(\Delta^{-})}$ does not hold because the  projective operators $P$ and $Q$ and thus $PQ$ and $QP$ are supported on some  neighbourhood of the diagonal, but the orthogonal projections $\pi_{\ker(\Delta^{\pm})}$  are by no means only supported on a small neighbourhood of the diagonal. This should justify the ad hoc definition of \eqref{eq:etasharp0}, although at the current stage it is not clear how to show such a relation in the projective case.
\end{remark}

Let $\D^L_{S^2 \times S^1,\proj}$  be the projective Dirac operator on  $S^2 \times S^1$ given by \eqref{eq:sharpprod0}. By Definition 2.5, its projective Atiyah-Patodi-Singer eta-invariant is 
\begin{equation}\label{eq:etasharp1}
\eta_{APS}\big(\D^L_{S^2 \times S^1,\proj} \big) := \ind(\D^{L,+}_{S^2,\proj})\cdot \eta_{APS}(\D_{S^1})
\end{equation}
where $\ind(\D^{L,+}_{S^2,\proj})$ is the projective analytical index in \eqref{eq:analprojind0} and $\eta_{APS}(\D_{S^1})$ denotes the usual eta-invariant of the ordinary Dirac operator on $S^1.$

\begin{corollary} \label{cor0}
	$\eta_{APS}\big(\D^L_{S^2 \times S^1,\proj} \big)=0.$
\end{corollary}
\begin{proof}
This follows from the fact that $\eta_{APS}(\D_{S^1})=0.$
\end{proof}

On the other hand, due to projectiveness, the kernel of $\D^L_{S^2 \times S^1,\proj}$ is not well-defined.  

\begin{assumption} Take $h(P \# A) := \dim \ker(A).$ 
\end{assumption}

\begin{definition}
	Let  $P,A$ and $R'$ as in Definition  ~\ref{etasharp}. Define the reduced eta-invariant of the projective Dirac operator $R'$ by 
	\begin{equation} \label{eq:redeta0}
	\bar{\eta}_{APS}(R')=\frac{\eta(R') + h(R')}{2} \text{ mod }\Z.
	\end{equation} 
\end{definition}

\begin{corollary}
Let $M_2 =S^1.$ Take  $P=\D^L_{S^2,\proj}$ and $A=\D_{S^1}.$ By Assumption 1,  we have
\begin{equation}
h(\D^L_{S^2 \times S^1,\proj}) = \dim \ker(\D_{S^1}) = 1
\end{equation}
and 
\begin{equation} \label{eq:redeta1}
\bar{\eta}_{APS}(\D^L_{S^2 \times S^1,\proj})=\frac{\eta(\D^L_{S^2 \times S^1,\proj}) + h(\D^L_{S^2 \times S^1,\proj})}{2} \text{ mod }\Z= \frac{1}{2}\text{ mod }\Z.
\end{equation} 
\end{corollary}

Combining the discussions above, we are now ready to state the result of this section.
\begin{theorem}
Let $X$ be a smooth compact manifold. Let $S^2$ be the Riemannian 2-sphere, together with a smooth map $f: S^2 \to X.$ Let $L$ be the Hermitian local line bundle whose normalised curvature is $B/2 \pi,$ defined by the pullback of a representative $\omega/2\pi$ in $H^2(X,\R/\Z)$ via $f.$ Let $\D^L_{S^2,\proj}$ be the projective Dirac operator twisted by $L$ on $S^2$ and let $\D_{S^1}$ be the usual Dirac operator on $S^1.$ Then, the analytic pairing $ H^2(X,\R/\Z) \times H_2(X) \to \R/\Z$ is given by 
\begin{equation} \label{eq:analpairing2}
\Big\langle \frac{\omega}{2\pi}, [\Sigma \xrightarrow{f} X] \Big\rangle = \bar{\eta}_{DZ}(\D^L_{S^2,\proj})- \int_{S^2} \frac{B}{2\pi} \text{ mod }\Z.
\end{equation}  Moreover, it is non-degenerate and well-defined.
\end{theorem}
\begin{proof}
From \eqref{eq:dzetaS2}, we consider the reduced Dai-Zhang eta invariant $\bar{\eta}_{DZ}(\D^L_{S^2,\proj})$ as the invariant $\bar{\eta}_{APS}(\D^L_{S^2 \times S^1,\proj})$  defined by Corollary \ref{cor0},  Assumption 1, and   \eqref{eq:redeta1}.	 Its justification has been given above, which follows from the extension to the cylinder and trivial gluing at both ends. The topological part (the second term) is the (modulo $\Z$) holonomy of the (normalised) curvature 2-form associated to the representative $L$ over $S^2.$ To show non-degeneracy, it suffices to show that 
\begin{equation}
H^2(S^2,\R/\Z) \to  \text{Hom}(H_2(S^2,\Z),\R/\Z) 
\end{equation} implemented  by the formula \eqref{eq:analpairing2} is an isomorphism. Notice that we are actually working on generators on both groups, i.e. the generator $L$ in $H^2(S^2,\R/\Z)$ and  the fundamental class $[S^2]$ in $H_2(S^2).$  So, the injectivity is implied. For surjectivity, it suffices to show that the map is non-zero, and thus sending generator to generator in $\R/\Z.$  Let $k \in \R$ be the integration of the topological term. Together with \eqref{eq:redeta1}, the pairing \eqref{eq:analpairing2}  reduces to  $1/2 - k$ modulo $\Z,$ which is non-zero in general. The isomorphism implies that the analytic pairing is non-degenerate.  The well-definedness follows as a special case of the analytic pairing in $\R/\Z$ $K^0$-theory with torsion twists, which will be proven elsewhere \cite{lim1}.   
\end{proof}


				\section{Analytic duality pairing $H^1(X,\R/\Z) \times H_1(X,\Z)$ }

		In this section, we study the analytic Pontryagin duality pairing in the cohomology of degree one, which is another phase calculation of the Aharonov-Bohm effect in Quantum Mechanics, cf. \cite{feynman1979feynman}.  Let $X$ be a smooth compact manifold.  By Fact 1, the group $H_1(X)$ is identified with the first oriented bordism group $\Omega^{or}_1(X),$ whose element is given by $[S^1 \xrightarrow{\gamma} X].$ 
		Then, the (classical) topological pairing  
		\begin{equation}
		H^1(X,\R/\Z) \times H_1(X) \to \R/\Z
		\end{equation} given  by 
		\begin{equation} \label{eq:H1topo0}
		\big(  A ,[S^1 \xrightarrow{\gamma} X] \big) \mapsto  \int_{S^1} \gamma^*A \text{ mod }\Z 
		\end{equation} is the holonomy of a (pullback) flat connection $A$ over a closed curve. 
		Apart from the classical pairing, there is also an analytic aspect.
	 	Let $\D_{S^1}=-id/d\theta$ be the usual Dirac operator on $S^1,$ with respect to the disconnected-cover spin structure, given by 
	 	\begin{equation}
	 	\tau =S^1 \times \C = \R \times \C / \sim,
	 	\end{equation}  where $(t,z) \sim (t',z')$ if and only if $t-t' \in \Z, z=z'.$ In other words, this is the `bad' spin structure of $S^1$ that does not extend to the disc $\mathbb{D}.$ 
		The group $H^1(X,\R/\Z)$ is usually interpreted as the set that classifies all of the isomorphic flat complex line bundles with connections over $X$ whose first Chern classes are torsion in $H^2(X,\Z).$  The pullback, via $\gamma,$ defines a flat complex line bundle over $S^1,$ which is necessarily trivial by a torsionality argument. More precisely, let 
		\begin{equation}\label{eq:flatlinebund0}
		L_\rho = \widetilde{X} \times_\rho U(1) 
		\end{equation} be the associated line bundle defined by a unitary representation $\rho : \pi_1(X) \to U(1).$ This bundle is flat and has the first Chern class  $c_1(L_\alpha) \in H^2_{tors}(X,\Z).$ Via $\gamma : S^1 \to X,$ we obtain the unitary representation $\rho'$ through the composition $$\rho' = \rho \circ \gamma_* : \pi_1(S^1) \to U(1),$$ which defines the flat line bundle 
		\begin{equation} \label{eq:flatlinebund1}
		\tilde{L}:= L_{\rho'} = \R \times_{\rho'} U(1)
		\end{equation} over $S^1.$ 
		A section of $\tilde{L}$ takes the form $f(\theta)v_{\rho'}(\theta),$ where  $f$ is a function on $S^1$ and $\nu_{\rho'}$ is a generating section given by 
		\begin{equation}
		\nu_{\rho'} (\theta) = \exp (2\pi i a \theta), \;\; a \in (0,1).
		\end{equation}
		Let $\D^{\tilde{L}}_{S^1}$ be the twisted-by-$\tilde{L}$ Dirac operator on $S^1.$ It is an ordinary self-adjoint elliptic differential operator. According to \cite{atiyah1975spectral2,gilkey2018invariance}, its eigenvalues are $\lambda_n=n+a,$ where $n$ is an integer obtained by differentiating $f.$ 
		Then, its Atiyah-Patodi-Singer eta-invariant is
		\begin{equation}
		\eta_{APS}(\D^{ \tilde{L}}_{S^1})=1-2a.
		\end{equation} which is non-zero in general, yielding the non-triviality of the eta-invariant.  Moreover, since $\dim \ker(\D^{ \tilde{L}}_{S^1})=1,$ the reduced eta-invariant is  
		\begin{equation}\label{eq:redeta2}
		\bar{\eta}_{APS}(\D^{ \tilde{L}}_{S^1}) = \frac{\eta_{APS}(\D^{ \tilde{L}}_{S^1}) + \dim \ker(\D^{ \tilde{L}}_{S^1})}{2}  = 1-a \text{ mod }\Z
		\end{equation}  which is again non-vanishing. 
		
		\begin{theorem}
		Let $X$ be a smooth compact manifold and let $\gamma : S^1 \to X$ be a loop. Let $\tilde{L}$ be the associated flat line bundle over $S^1$ defined by \eqref{eq:flatlinebund1} via $\gamma$. Let $\D^{\tilde{L}}_{S^1}$ be the corresponding twisted Dirac operator. Then, the analytic  pairing $H^1(X,\R/\Z) \times H_1(X) \to \R/\Z$ is given by 
		\begin{equation}
		\langle A ,[S^1 \xrightarrow{\gamma} X] \rangle=\bar{\eta}_{APS} (\D^{\tilde{L}}_{S^1}) - \int_{S^1} \gamma^*A \text{ mod } \Z.
		\end{equation} This pairing is well-defined and non-degenerate.
		\end{theorem}   
		
		\begin{proof}
		The validity and non-triviality of the analytic part of \eqref{eq:redeta2} are discussed above. The topological term is the (reduced modulo $\Z$) holonomy of a flat connection $A$ over a closed curve. This pairing formula is a special case of the analytic pairing in $\R/\Z$ $K^1$-theory, cf. \cite[Proposition 3]{lott1994r}. In particular, the well-definedness and the non-degeneracy follow from the general case. For instance, the pullback data defines a triple $(\tilde{L},\nabla^{\tilde{L}},\omega)$ over $S^1,$ where $\omega$ is a 1-form satisfying  
		\begin{equation}
		d\omega=c_1(\tilde{L},\nabla^{\tilde{L}})=0.
		\end{equation} From a standard calculation of the curvature $F_{\nabla^{\tilde{L}}}=\nabla^{\tilde{L}}\circ \nabla^{\tilde{L}}=dA',$ where $A'=\gamma^*A,$  we see that  $\omega$ is cohomologous to $A'.$ By Stokes theorem, the topological integration is independent of the choice of 1-form.  The others are routine work and  are left to the reader.
		\end{proof}
		


				\section{The even $\R/\Z$ $K$-theory  }
	
				In this section, we give a model of the even $K$-theory with coefficients in $\R/\Z.$ We show that the proposed geometric model is a $K^0(X)$-module and has a well-defined  $\R/\Q$ Chern character map. In \cite{basu2005k}, Basu gave a model of this group by considering the suspension of the $\R/\Z$ $K^1$-theory, whose cocycle is a pair of vector bundles $(E_1,E_2)$ over the suspension $SX,$ together with an isomorphism $\phi: E_1 \otimes V \cong E_2 \otimes V,$ where $V$ is a von-Neumann algebra bundle.  However, this is not an appropriate model in formulating the analytic $K^0$-pairing. The main reason is that differential forms are used in a fundamental way, but the suspension $SX$ may not be smooth even if $X$ is smooth. Moreover, the model given below is compatible with the construction of the Dai-Zhang eta-invariant, which is the analytic term \eqref{eq:K0pairing} of the analytic $K^0$-pairing and thus justifying the claim that it is a valid and a direct analog to the Lott's analytic $K^1$-pairing \cite{lott1994r}.

\subsection{The group $K^0(X,\R/\Z)$}

\begin{definition}Let $X$ be a smooth compact manifold. An $\R/\Z$ $K^0$-cocycle over $X$ is a triple 
	\begin{equation}\label{eq:rzK0cocycle}
	(g,(d,g^{-1}dg),\mu)
	\end{equation}
	 where 
	\begin{itemize}
		\item $g: X \to U(N)$ is a smooth map, i.e. a $K^1(X)$-representative,
		\item $(d,g^{-1}dg)$ is a pair of flat connections on the trivial bundle defined by $g,$
		\item $\mu  \in  \Omega^{\text{even}}(X) /d\Omega$ satisfying the \textit{exactness} condition
		\begin{equation}\label{eq:quantization1}
		d\mu=ch(g,d)-\text{Tr}(g^{-1}dg).
		\end{equation}
	\end{itemize} 
\end{definition} 

Here, the odd Chern character of $g$ with flat connections $(d,g^{-1}dg)$ is explicitly given by 
\begin{equation} \label{eq:oddch0}
ch(g,d):=\sum^{\infty}_{n=0} \frac{n!}{(2n+1)!} \text{Tr}(g^{-1}dg)^{2n+1}
\end{equation} cf. \cite{getzler1993odd} and \cite{zhang2001lectures}.

\begin{definition} 
	Let $g_i : X \to U(N_i),$ for $i=1,2,3,$ be smooth maps for large $N_i \in \Z$. Let $\cE_i$ be the $\R/\Z$ $K^0$-cocycles corresponding to $g_i.$ Then, the $\R/\Z$ $K^0$-relation is given by 
	\begin{equation} \label{eq:k0relation0}
	\cE_2=\cE_1 +\cE_3,
	\end{equation} i.e. whenever there is a sequence of maps $g_1 \longrightarrow g_2 \longrightarrow g_3$ such that 
	\begin{equation}
	g_2 \simeq g_1 \oplus g_3,
	\end{equation} which can be viewed as $g_2$ being homotopic to $\text{Diag}(g_1,g_3)$ as unitary matrices, then  
	\begin{equation}\label{eq:k0rel}
	\mu_2=\mu_1+\mu_3 -\text{T}ch(g_1,g_2,g_3).
	\end{equation}
	Here, $\text{T}ch(g_1,g_2,g_3)$ denotes the transgression form of the odd Chern character satisfying  
	\begin{equation}
		d\text{T}ch(g_1,g_2,g_3)= ch(g_1)-ch(g_2)+ch(g_3).
	\end{equation}
\end{definition}

\begin{remark}
	The transgression  $\text{T}ch(g_1,g_2,g_3)$ is taken as $\text{T}ch((i\oplus j)^*g_2, g_1 \oplus g_3)$  where $i :g_1 \to g_2$ is the inclusion and $j:g_3 \to g_2$ is a splitting map.  The term with two entries  is the transgression form of the odd Chern character defined by 
	\begin{equation}\label{eq:oddtransg1}
	\text{T}ch(g_t,d)=\sum^{\infty}_{n=0} \frac{n!}{(2n)!} \int^1_0 \tr \Big( g^{-1}_t \frac{\partial g_t}{\partial t}(g^{-1}dg_t)^{2n} \Big) dt
	\end{equation} where $g_t$ is a path of smooth maps joining $(i\oplus j)^*g_2$ and $g_1 \oplus g_3,$  for $0<t<1.$  One can show that   $\text{T}ch((i\oplus j)^*g_2, g_1 \oplus g_3)$ is independent of the choice of  $j.$
\end{remark}

\begin{definition}
	Let $X$ be a smooth compact manifold. The $\R/\Z$ $K^0$-theory of $X,$ denoted by $K^0(X,\R/\Z),$ consists of all $\R/\Z$ $K^0$-cocycles with zero virtual trace in the lowest degree, modulo the $\R/\Z$ $K^0$-relation. The group operation is given by the addition of $\R/\Z$ $K^0$-cocycles
	\begin{equation}(g,(d,g^{-1}dg),\mu)+(h,(d,h^{-1}dh),\theta)=(g\oplus h,(d\oplus d,g^{-1}dg \oplus h^{-1}dh),\mu \oplus \theta).
	\end{equation}
\end{definition}

\begin{remark} \label{altK1}
	There is another equivalent definition of $K^1(X),$ in which a class is represented by a pair $(E,h)$ where $E$ is a complex vector bundle over $X$ and $h$ is a smooth automorphism of $E.$ One way to see the equivalence between these two definitions of $K^1(X)$ is by first complementing $E$ to a trivial bundle $\tau$ by a complementary bundle $E^c,$ which  always exists. Let $T$ be an isomorphism $E \oplus E^c \cong \tau.$ Let $\tilde{g}:= T^{-1}(h \oplus \text{Id}_{E^c})T$ be an automorphism of $\tau.$  Then, $\tilde{g}$ and $h$ define the same class in $K^1(X).$  
	
	Note that the second entry of \eqref{eq:rzK0cocycle} is uniquely determined by $g.$ However, if another definition of $K^1(X)$ is used, then a choice of a pair of connections comes into the picture. In particular, the cocycle  $(g,(d,g^{-1}dg),\mu)$ can be equivalently modified to $(h, (\nabla^E, h^{-1}\circ \nabla^E\circ h),\mu)$ for a pair $(E,h)$ where $E$ is a complex vector bundle with connection $\nabla^E,$ $h$ is an automorphism of $E,$ $(\nabla^E,h^{-1}\nabla^E h)$ is a pair of connections on $E$ and $\mu$ is an even degree form on $X$ satisfying the exactness condition. The relation is similar: whenever there is a SES of maps $0\to h_1 \to h_2 \to h_3 \to 0,$  the relation is given by $\xi_2=\xi_1 + \xi_3.$  
	The corresponding odd Chern character of $(E,h)$ is defined by  
	\begin{equation}
	ch(h):=CS(\nabla^E,h^{-1}\circ \nabla^E \circ h).
	\end{equation}
	 Its explicit formula is now in the general setting and becomes more complicated, see \cite{fedosov1996index}. 
\end{remark}

\textbf{$K^0(X)$-module structure:}   We show that the group $K^0(X,\R/\Z)$ is a $K^0(X)$-module. For clarity, we use the second definition of $K^1(X)$ as in Remark ~\ref{altK1}.

Let $(E,g)$ be an $K^1$-representative. The module multiplication
$$ K^0(X) \times K^0(X,\R/\Z) \longrightarrow K^0(X,\R/\Z)$$ is given by
\begin{equation} \label{eq:K0module}
V \hat{\otimes} (g,(\nabla^E,g^{-1}\nabla^E g),\mu) = \big(V \otimes E,(\nabla^V \otimes \nabla^E, h^{-1}\nabla^V h \otimes g^{-1}\nabla^E g), ch(\nabla^V)\wedge \mu \big)
\end{equation} where $h$ is a chosen automorphism of $V.$

The tensor product \eqref{eq:K0module} requires a choice of automorphism $h$ of $V,$ which   
always exists from the viewpoint of the complementary bundle and the automorphism of the trivial bundle as a global trivialisation. Here, $\nabla^V \otimes \nabla^E := \nabla^V \otimes 1 + 1 \otimes \nabla^E.$ 
Fix $g,$ consider two choices $h_1$ and $h_2$ so that
\begin{equation} \label{eq:chtensor0}
ch(h_1 \otimes g ) = CS(\nabla^V \otimes \nabla^E, h^{-1}_1 \nabla^V h_1 \otimes g^{-1} \nabla^E g) 
\end{equation}
\begin{equation} \label{eq:chtensor1}
ch(h_2 \otimes g ) = CS(\nabla^V \otimes \nabla^E, h^{-1}_2 \nabla^V h_2 \otimes g^{-1} \nabla^E g).
\end{equation}
By taking the difference \eqref{eq:chtensor0} -- \eqref{eq:chtensor1}, we get 
\begin{align*}
ch(h_1 \otimes g ) &- ch(h_2 \otimes g )  \\
&=ch(g^{-1}\nabla^E g ) \wedge \big( CS(\nabla^V ,h^{-1}_1 \nabla^V h_1) - CS(\nabla^V ,h^{-1}_2 \nabla^V h_2)\big) \\
&= ch(g^{-1}\nabla^E g ) \wedge CS(h^{-1}_2 \nabla^V h_2, h^{-1}_1 \nabla^V h_1)
\end{align*} 
If $h_1$ and $h_2$ represent the same class, then $h_2h^{-1}_1$ is homotopic to the identity. The Chern-Simons form reduces to  $CS(\nabla^V , \nabla^V).$ For $t\in [0,1],$ let $\gamma(t)$ be a path of connections joining $\nabla^V$ back to itself, which is a closed curve. Let $A_t \in \Omega^1(X,\text{End}(V))$ and $R_t$ be the curvature of $\nabla^V_t.$ Consider $$cs(\gamma)=\int_0^1 \sum_{j=1}\frac{1}{(j-1)!} \Big( \frac{1}{2\pi i}\Big)^j \text{Tr}(A_t \wedge (R_t)^{j-1}).$$ By \cite[\text{Proposition} 1.6]{simons2008structured}, the odd form $cs(\gamma)$ is exact since $\gamma$ is a closed curve. Together with \cite[(1.7)]{simons2008structured}, we have $$CS(\nabla^V,\nabla^V)=cs(\gamma) \text{ mod exact} \equiv 0. $$ The above argument shows the following corollary. 

\begin{corollary}
	For a fixed $K^0$-cocycle in $K^0(X,\R/\Z)$, the module multiplication given by \eqref{eq:K0module} only depends on the homotopy class of $h.$
\end{corollary}

Moreover, since $ch(\nabla^V)$ is closed, it is straightforward that  
\begin{equation}
d(ch(\nabla^V) \wedge \mu)=ch(\nabla^V) \wedge d\mu.
\end{equation} 

 \begin{remark} There is also a description using $\Z_2$-graded cocycles.
 A $\Z_2$-graded $K^0$-cocycle consists of $(g^{\pm}, (d, g^{\pm}dg^{\pm}),\mu)$ where $g^{\pm}=g^+ \oplus g^-$ is a $\Z_2$-graded $K^1$ element and $\mu \in \Omega^{\text{even}}(X)/d\Omega$ such that 
 \begin{equation}
 d\mu=ch(g^{\pm},d)=ch(g^+,d) - ch(g^-,d).
 \end{equation}    
 \end{remark} 

\textbf{Explicit maps and exactness of (part of) sequence :} Consider the  sequence
\begin{equation} \label{eq:les1}
\cdots \to K^0(X,\R) \xrightarrow{\alpha} K^0(X,\R/\Z) \xrightarrow{\beta} K^1(X,\Z) \xrightarrow{ch} K^1(X,\R) \to \cdots 
\end{equation}
associated to the short exact sequence of coefficients $1 \to \Z \to \R \to \R/\Z  \to 1,$ where  
\begin{align}
& \alpha(\mu)=(\id, (d,d),\mu) - (\id,(d,d),0)=(0,0,\mu) \text{  is the inclusion map,}   \label{eq:alpha0}\\
& \beta(g,(d,d+g^{-1}dg),\theta )=[g] \text{  is the forgetful map,} \label{eq:beta0} \\
& ch(g) \text{  is the odd Chern character map given by } \eqref{eq:oddch0}.
\end{align}

\begin{lemma}\label{exactK0}
	With respect to the sequence \eqref{eq:les1}, it is exact at $K^0(X,\R/\Z)$ and at $K^1(X,\Z).$
\end{lemma}

\begin{proof}
	For $$K^0(X,\R) \xrightarrow{\alpha} K^0(X,\R/\Z) \xrightarrow{\beta} K^1(X,\Z),$$ note that $\im(\alpha) \subseteq \ker(\beta)$ follows from the definition. We need to show $\ker(\beta) \subseteq \im(\alpha).$ Let $\cE_1-\cE_2=(g_1, (d,g^{-1}_1dg^{-1}_1),\mu_1)-(g_2, (d,g^{-1}_2dg^{-1}_2),\mu_2) \in \ker(\beta)$ so that $\beta(\cE_1-\cE_2)=[g_1]-[g_2]=0.$ In particular, $[g_1]=[g_2]$ if and only if there exists an identity matrix $\id$ of suitable rank in the unitary group, such that $g_1 \oplus \id $ is homotopic to $g_2 \oplus \id$. The direct sum means that  they sit along the diagonal of a suitably large matrix $h.$    This defines an element $(h, (d,h^{-1}dh), \mu_1) -(h, (d,h^{-1}dh), \mu_2)=(0,0, \mu_1-\mu_2)$  in $\im(\alpha),$ as the image of $\mu_1-\mu_2$ under $\alpha.$ This shows $ker(\beta) \subseteq im(\alpha)$ and thus is exact at $K^0(X,\R/\Z).$
	
	On the other hand, consider $$K^0(X,\R/\Z) \xrightarrow{\beta} K^1(X,\Z)  \xrightarrow{ch} K^1(X,\R).$$ Since any $[g]$ in $K^1(X)$ with vanishing Chern character lies in the torsion subgroup $K^1_{tors}(X),$ the sequence reduces to $$K^0(X,\R/\Z) \xrightarrow{\beta} K^1_{tors}(X) \to 0.$$ Hence, an element in $K^1_{tors}(X)$ lifts to an element in $K^0(X,\R/\Z)$ such that it is the image under $\beta.$ This shows  $ker(ch) \subseteq im(\beta) .$ To show the opposite direction, consider two elements $\cE_1$ and $\cE_2$ in $K^0(X,\R/\Z).$ By applying the odd Chern character to the image of $\beta,$ together with the exactness condition, we get  $ch([g_1]-[g_2])=[d(\mu_1 -\mu_2)]=0.$ So, $[g_1]-[g_2]$ lies in the kernel of $ch.$ This shows $im(\beta) \subseteq ker(ch)$ and thus is exact at $K^0(X,\Z).$

\end{proof}

			\subsection{The $\R/\Q$ Chern character $ch_{\R/\Q}$}
			Next, we define the $\R/\Q$ Chern character map $ch_{\R/\Q}$ between   $K^0(X,\R/\Z)$ and  $H^{\text{even}}(X,\R/\Q)$ such that  the following diagram commutes.

\begin{center}
	\begin{tikzcd}
		\cdots \to    K^0(X,\mathbb{R}) \arrow[d,"\cong"] \arrow[r,"\alpha"]  & K^0(X, \mathbb{R}/\mathbb{Z})  \arrow[r,"\beta"] \arrow[d,"\text{ch}_{\mathbb{R}/\mathbb{Q}}"] &	K^1(X,\Z)\arrow[r,"ch"] \arrow[d,"\text{ch}_{\mathbb{Q}}"] & K^1(X,\mathbb{R})\arrow[d,"\cong"]  \to \cdots \\
		\cdots \to   H^{even}(X,\mathbb{R})  \arrow[r,"r"]   & H^{even}(X,\mathbb{R}/\mathbb{Q})\arrow[r,"\widetilde{\beta}"]  & H^{odd}(X,\mathbb{Q}) \arrow[r,"i"]  & H^{odd}(X,\mathbb{R}) \to \cdots
	\end{tikzcd}
\end{center}


The upper (resp. bottom) row is the long exact sequence of $K$-theory (resp. cohomology) associated to the short exact sequence of the coefficients. Here $r, i$ and $\widetilde{\beta}$ are the reduction, inclusion and Bockstein maps in cohomology respectively. The maps in the upper row are given by \eqref{eq:alpha0}, \eqref{eq:beta0} and the odd Chern character.  
By tensoring the upper row by $\Q$ and by applying the Five lemma, $ch_{\R/\Q}$ is a rational isomorphism.

Now, the existence of $\mu$ in $K^0(X,\R)$ implies that the odd Chern character $ch(g-\id_N)=0,$ where $\id_N$ denotes the identity matrix of size $N\times N$ with respect to $g:X \to U(N),$ for some large $N\in \Z.$  So, $g-\id_N$ is torsion in $K^1(X)$ and there exists some positive $k$ such that $kg \cong \id_{kN},$ i.e. $g \oplus \cdots \oplus g =\text{Diag}(g,...,g)$ is homotopic to the identity matrix. Using the second definition of $K^1,$ i.e. by viewing $g$ as a smooth automorphism of a complex vector bundle $E,$ the unitary map $kg$ corresponds to an automorphism on $kE=E \oplus \cdots \oplus E.$ Let $k\nabla^E$ be its Hermitian connection and $\nabla^{kE}_0$ be a connection with trivial holonomy. Then, we obtain the conjugation $h^{-1}k\nabla^E h$ and $h^{-1}\nabla^{kE}_0 h$ by $h=kg$ of these two connections. For $t \in [0,1],$ fix $k\nabla^E$ and $\nabla^{kE}_0$  and vary $h$ within the homotopy class of $g,$ giving a path $h(t)$ connecting ${h(t)}^{-1}k\nabla^E h(t)$ and ${h(t)}^{-1}\nabla^{kE}_0 h(t).$  This defines $$ch(h(t),t\in [0,1]) \in \Omega^{\text{odd}}(X \times [0,1]).$$ 
By the standard construction in \cite{getzler1993odd,zhang2001lectures}, the respective transgression form is
$$\text{T} ch(h(t),[0,1]) = \varphi \int^1_0 \tr \Big( {h(t)}^{-1}\frac{\partial (h(t))}{\partial t} \Big( {h(t)}^{-1}(k\nabla^E) h(t) \Big)^{2n} \Big) dt.$$
This is an analog of \eqref{eq:oddtransg1}. Then, 
\begin{equation} \label{eq:rqch1}
\frac{1}{k}\text{T}ch(h(t),t\in [0,1]) -\mu
\end{equation}
defines an element in $H^{\text{even}}(X,\R).$

\begin{definition}
	Let $ch^0_{\R/\Q} (g,(d,g^{-1}dg), \mu) $  be the image of $  \frac{1}{k}\text{T}ch(h(t),t\in [0,1]) -\mu $
	under the map $H^{\text{even}}(X,\R)\to H^{\text{even}}(X,\R/\Q).$ 
\end{definition}  

Next, we show that $ch^0_{\R/\Q} (g,(d,g^{-1}dg), \mu)$ is well-defined.

\begin{lemma}
	Let $\cE=(g,(d,g^{-1}dg), \mu).$ As an image in $H^{\text{even}}(X,\R/\Q),$ $ch^0_{\R/\Q} (\cE)$ is independent of the choice of the homotopy class of $h$ and the choice of $k.$
\end{lemma}

\begin{proof}
	Let $g_1,g_2$ be two $K^1$ elements. Let $h_1(t)$ and $h_2(t)$ be the respective paths as constructed above. That is, $h_i(t)$ connects $h^{-1}_i k\nabla^E h_i$ and $h^{-1}_i\nabla^{kE}_0 h_i$ for $i=1,2.$ Simply denote their transgression forms by $\text{T}ch(h_1(t))$ and $\text{T}ch(h_2(t))$ respectively. Note that in general $h_1(t)$ and $h_2(t)$ may not coincide, in which case both paths lie within their homotopy class. However, it is possible to connect $h_1(t)$ and $h_2(t)$ at the left endpoint. Since both $h_1$ and $h_2$ are unitary, we consider the multiplication $h_2^{-1}h_1$ for a fixed $k.$ Then, the two left endpoints can be joined by the conjugation of  $h_2^{-1}h_1$ since 
	$$(h^{-1}_2 h_1)^{-1}(h^{-1}_2 \circ (k\nabla^E) \circ h_2) (h^{-1}_2 h_1) = h^{-1}_1 h_2 h^{-1}_2 \circ (k\nabla^E) \circ h_2 h^{-1}_2 h_1=h^{-1}_1 \circ (k\nabla^E) \circ h_1.$$
	Let $r(h_2^{-1} h_1)$ be the conjugation action. For $t\in [0,1],$ define 
	$$(h_1 h_2^{-1})(t):= h_1(t) \circ r(h_2^{-1} h_1) \circ h_2(t)^{-1}.$$ Then, the difference 
	\begin{equation}\label{eq:rqconndiff1}
	\frac{1}{k}\text{T}ch(h_1(t)) - \frac{1}{k}\text{T}ch(h_2(t))  \\ 
	= \frac{1}{k}\text{T}ch((h_1 h_2^{-1})(t)) +d\omega_n 
	\end{equation}  
	holds, where the second term of the RHS of \eqref{eq:rqconndiff1} is some exact form independent of $h_i,$  c.f. \cite[Corollary 1.18]{zhang2001lectures}. In particular, the difference \eqref{eq:rqconndiff1} is the same up to multiplication by a rational number, as the image of $ch([h_1][h_2^{-1}])=ch([h_1])\wedge ch([h_2^{-1}]) \in H^{\text{even}}(X,\Q)$ in $H^{\text{even}}(X,\R),$ so it vanishes when mapped into $H^{\text{even}}(X,\R/\Q).$ This shows that $ch^0_{\R/\Q} (\cE)$ is independent of the homotopy class of $h.$

	Next, for two different positive integers $k$ and $k',$ while keeping the choice of $g$ fixed, we get $h(t)=kg_t$ and $h'(t)=k'g_t.$ Then, the difference is 
	\begin{equation*}
	\frac{1}{k}\text{T}ch(h(t)) - \frac{1}{k'}\text{T}ch(h'(t)) 
	= \frac{1}{kk'} \text{T}ch((h' h^{-1})(t)) +d\omega_n.
	\end{equation*}  
	By a similar argument as in the previous paragraph, the difference is the same up to multiplication by a rational number, as the image of the odd Chern character in $H^{\text{even}}(X,\R/\Q)$ vanishes. So, the image of $ch^0_{\R/\Q}$ is independent of the positive integer $k.$
	
\end{proof}


				\section{Analytic duality pairing $K^0(X,\R/\Z) \times K_0(X)$}
In this section, we explain the formulation of the analytic  pairing in $\R/\Z$ $K^0$-theory by applying the Dai-Zhang eta-invariant. This main result of this paper is the following theorem. A detailed proof is given in the next section.

 \begin{theorem}\label{nondegenpair0}
	Let $M$ be an even dimensional closed $\spinc$ manifold and let $E$ be a complex vector bundle over $M.$ Let $X$ be a smooth compact base manifold , together with a smooth map $f:M \to X.$ Let $h=g \circ f: M \to U(N)$ be an $K^1$-element of $M$ and let $\tau$ be the trivial bundle in which $h$ acts as an automorphism. Let $\D^{\psi,h}_{E \otimes \tau,M \times [0,1]}$ be the Dirac operator twisted by $E$ and $\tau$ on the cylinder $M \times [0,1],$ defined by 
	\begin{equation}
	\D^{\psi,h}_{E \otimes \tau, M \times [0,1]}= \D_{E\otimes \tau} + (1-\psi)h^{-1}[\D_{E\otimes \tau},h].
	\end{equation} 
	Let $\bar{\eta}(\D^{\psi,h}_{E\otimes \tau, M \times [0,1]})$  be its reduced eta-invariant. Then, the analytic pairing 
	$$ K^0(X,\R/\Z) \times K_0(X) \longrightarrow \R/\Z$$ given by 
	\begin{align}
	\langle (g,(d,g^{-1}&dg),\mu),(M,E,f)  \rangle  \nonumber  \\
	= \;&\bar{\eta}\big( \D^{\psi,h}_{E \otimes \tau, M \times [0,1]} \big) -\int_M f^*\mu \wedge ch(E)\wedge \Td(M) \text{ mod } \Z \label{eq:K0pairing}
	\end{align} is well-defined and non-degenerate. 
\end{theorem}

\subsection{Explanation of related terms}
Let $X$ be a smooth compact manifold. The Baum-Douglas even geometric $K$-homology, denoted as $K_0(X):=K_0(X,\Z),$ is a well-known extraordinary homology theory associated to $X$  defined by geometric data:
\begin{definition} \cite{baum1982k} \label{BDKhom0}
 A geometric $K_0$-cycle over $X$ is a triple $(M,E,f)$ where $M$ is an even dimensional closed $\spinc$ manifold, $E$ is a complex vector bundle over $M$ and $f:M \to X$ is a smooth map. The group $K_0(X)$ is generated by all isomorphic  $K_0$-cycles modulo the three relations: direct sum-disjoint union, vector bundle modification and bordism. 
\end{definition}

\begin{remark}
For simplicity, we assume $M$ to be connected. 
The $K_0$-homology is a $K^0(X)$-module where the cap product is given by  
\begin{equation}
V \cap (M,E,f) \mapsto (M, f^*V \otimes E ,f)
\end{equation} for a complex vector bundle $V$ over $X.$ 
Moreover, it has a well-defined Chern character map in $K$-homology
\begin{equation}
ch(M,E,f) = f_*[ch(E) \wedge \Td(M)]  \; \in H^{\text{even}}(X,\Q). 
\end{equation}
\end{remark}

In the following, we  explain the analytic term of \eqref{eq:K0pairing}.
At this stage, we shall call it the reduced Dai-Zhang eta-invariant. We sketch its construction following \cite{dai2006index} and relate it with the analytic pairing using a $K_0$-cycle $(M,E,f)$   and a $\R/\Z$ $K^0$-cocycle $(g,(d,g^{-1}dg),\mu).$  

Let $M$ be an even dimensional closed $\spinc$ manifold, $E$ be a complex vector bundle over $M,$ and $f:M \to X$ a smooth map where $X$ is a smooth compact manifold. Let $\D_{E,M}$ be the Dirac operator on $M$ twisted by $E.$ 
\begin{itemize}
	\item Twist $\D_{E,M}$ by $h=g\circ f :M \to U(N),$ defined on $L^2(S\otimes  E \otimes \tau)$ by acting as the identity on $L^2(S\otimes E)$ and $h$ acts as an automorphism on $\tau.$ Denote this by  $\D^h_{E\otimes \tau,M};$
	\item Extend $S\otimes  E \otimes \tau$ trivially to the cylinder $[0,1]\times M$ equipped with a product metric, i.e. over each $t \in [0,1]$ there is a copy of $E.$ Let $\psi=\psi(t)$ be a cut-off function on $[0,1]$ which is identically 1 in a $\epsilon$-neighbourhood of $M$ for small $\epsilon >0$ and 0 outside of a $2\epsilon$-neighbourhood of $M.$ Consider the Dirac-type operator 
	$$\D^{\psi}_{E\otimes \tau,M \times [0,1]}=(1-\psi)\D_{E\otimes \tau}+\psi h\D_{E\otimes \tau} h^{-1}$$ and take its conjugation 
	\begin{equation}
	\D^{\psi,h}_{E\otimes \tau,M \times [0,1]}=h^{-1}\D^{\psi}_E h = \D_{E\otimes \tau}+ (1-\psi) h^{-1}[\D_{E\otimes \tau}, h];
	\end{equation} 
	\item Assume that the Lagrangian $L \subset \ker(\D^h_{E\otimes \tau,M})$ exists and fix a choice,  and equip one end $M \times \{0\}$ with the modified Atiyah-Patodi-Singer boundary conditions 
	\begin{equation}
	P^\partial = P_{\geq 0} + P_L :  L^2_{\geq 0}(S \otimes E\otimes \tau) \to L^2_{\geq 0}(S\otimes E\otimes \tau |_M) \oplus L
	\end{equation} 
	where $P_{\geq 0}$ is the Atiyah-Patodi-Singer boundary projection \cite{atiyah1975spectral1}. Equip the Dirac-type operator on the other end $M \times \{1\}$ with $\id - h^{-1}P^\partial h.$ 
\end{itemize} 
Then, $(\D^{\psi,h}_{E\otimes \tau,M \times [0,1]}, P^\partial,\id - h^{-1}P^\partial h )$ is a self-adjoint elliptic boundary problem. For simplicity, we denote the boundary problem by $\D^{\psi,h}_{E\otimes \tau,M \times [0,1]},$ i.e. with the boundary conditions implicitly implied. Let the eta-function of $\D^{\psi,h}_{E\otimes \tau,M \times [0,1]}$ be given by the usual formula
\begin{equation}
\eta(\D^{\psi,h}_{E\otimes \tau,M \times [0,1]},s)=\sum_{\lambda \neq 0} \frac{sgn(\lambda)}{|\lambda|^s}
\end{equation}
 for $\text{Re}(s)$ sufficiently large and the sum runs through all non-zero eigenvalues $\lambda$ of $\D^{\psi,h}_{E\otimes \tau,M \times [0,1]}.$ Take $\eta(\D^{\psi,h}_{E\otimes \tau,M \times [0,1]}):= \eta(\D^{\psi,h}_{E\otimes \tau,M \times [0,1]},0).$ Let $\hat{\eta}(\D^{\psi,h}_{E\otimes \tau,M \times [0,1]})$ be the full  eta-invariant defined by 
\begin{equation}\label{eq:partialredeta1}
\widehat{\eta}(\D^{\psi,h}_{E\otimes \tau,M \times [0,1]}) = \frac{\eta (\D^{\psi,h}_{E\otimes \tau,M \times [0,1]}) + h(\D^{\psi,h}_{E\otimes \tau,M \times [0,1]})}{2}
\end{equation} where $h(\D^{\psi,h}_{E\otimes \tau,M \times [0,1]}) = \dim \ker(\D^{\psi,h}_{E\otimes \tau,M \times [0,1]}).$

\begin{definition}[\cite{dai2006index}] With the construction above, define an eta-type invariant on an even dimensional closed manifold by 
	\begin{equation}\label{eq:DZeta1}
	\widehat{\eta}(M,E,h)=\widehat{\eta}(\D^{\psi,h}_{E\otimes \tau,M \times [0,1]}) - \text{sf}\big\{ \D^{\psi,h}_{E\otimes \tau,M \times [0,1]}(s) ; s \in [0,1] \big\}
	\end{equation}
	where the second term is the spectral flow of $\D^{\psi,h}_{E\otimes \tau,M \times [0,1]}(s)$ given by 
	\begin{equation}\label{eq:familyD1}
	\D^{\psi,h}_{E\otimes \tau,M \times [0,1]}(s)=\D_{E\otimes \tau} + (1-s\psi)h^{-1}\D_{E\otimes \tau} h
	\end{equation} on $M \times [0,1]$ with boundary conditions  $P^\partial$ on $M \times \{0\}$ and $\id -h^{-1}P^\partial h$ on $M \times \{1\}.$ That is, \eqref{eq:familyD1} is a path connecting $h^{-1} \D_{E\otimes \tau} h$ and $\D^{\psi,h}_{E \otimes \tau,M \times [0,1]}.$ We call \eqref{eq:DZeta1} the Dai-Zhang eta-invariant. 
\end{definition}

\begin{remark}
	The spectral flow is a priori an integer (c.f. \cite{atiyah1976spectral3}), measuring the net change between the positive crossing (from negative to positive eigenvalues across 0) and the negative crossing (from positive to negative eigenvalues across 0). Upon reducing  \eqref{eq:DZeta1} modulo $\Z,$ we obtain an $\R/\Z$-valued spectral invariant
	\begin{equation}\label{eq:DZeta2}	
	\bar{\eta}(\D^{\psi,h}_{E \otimes \tau, M\times [0,1]}):= \widehat{\eta}(M,E,h) \text{ mod }\Z \equiv \widehat{\eta}(\D^{\psi,h}_{E \otimes \tau, M\times [0,1]}) \text{ mod }\Z.
	\end{equation}
\end{remark}

\begin{remark}
As shown in \cite{dai2006index}, the invariant $\bar{\eta}(M,E,h)$ is independent of the cut-off function $\psi.$ Moreover, by reducing modulo $\Z,$ the invariant $\bar{\eta}(M,E,h)$ is independent of the length $a >0$ of the cylinder $M \times [0,a].$ Recall that the construction of such an eta-invariant requires  the modified Atiyah-Patodi-Singer boundary conditions at both ends of the cylinder. It is a remarkable fact that it holds for more general $Cl(1)$-spectral sections  $P$ and is independent of the choice of such spectral sections, see \cite[\text{Prop} 5.6]{dai2006index}. Hence, the expression of  eta-invariant \eqref{eq:DZeta2} is valid, as it only depends on the variables $(M,E,h),$ i.e. the underlying geometry(metric and connection) of $M$ and $E$ and a choice of $h,$ regardless of other variables used in the construction. 
\end{remark}

The reduced Dai-Zhang eta invariant \eqref{eq:DZeta2} defines the analytic term of the pairing \eqref{eq:K0pairing}. On the other hand, the topological term is exactly the (reduced modulo $\Z$) integration of the pullback of some even form from $X$ and local characteristic forms on $M,$ 
\begin{equation}
\int_M f^*\mu \wedge ch(E,\nabla^E) \wedge \Td(M) \text{ mod }\Z.
\end{equation} In general, they are not mutually exclusive: the intertwined relation between these two parts lies in the exactness condition \eqref{eq:quantization1}. This summarises the explanation of the analytic pairing formula \eqref{eq:K0pairing}.


\section{Proof of Theorem ~\ref{nondegenpair0}}
This section is devoted to show the well-definedness and the non-degeneracy of \eqref{eq:K0pairing}.

\subsection{Well-definedness of $K^0$ pairing}
We show that the pairing \eqref{eq:K0pairing} is independent of the underlying geometry of the manifold and the vector bundle,  respects the Baum-Douglas $K$-homology relation and respects the $\R/\Z$ $K^0$-relation defined above.  

\subsubsection{Well-defined on the level of cycle}

\begin{proposition} \label{k0wd1}
	The analytic pairing \eqref{eq:K0pairing} is independent of the Riemannian metric of the manifold $M,$  the Hermitian metric and the connection on the complex vector bundle $E.$ 
\end{proposition}

\begin{proof}
Fix a $K^0$-cocycle $(r,(d,r^{-1}dr),\mu)$ and  a $K_0$-cycle $(M,E,f).$  For $i=1,2,$ let  $M_i=(M,g_i)$ be the same even dimensional manifold with different Riemannian metrics $g_i,$ which is the boundary of a cylinder $N= M \times [0,1],$ i.e. $\partial N \cong M_1 \sqcup -M_2.$ Let 
\begin{equation}
g_\gamma =\gamma(t) +(dt)^2
\end{equation} 
 be the extended metric on $N,$ where $\gamma(t)$ is a path in the space of Riemannian metrics on $M.$ Let $g^{E_i},\nabla^{E_i}$ be the metric and Hermitian connection on $E_i=E|_{M_i}$ respectively. Let $\tau$ be the trivial bundle in which a fixed $K^1$-element $h=r \circ f : M \to U(N)$ acts as an automorphism. Let $\nabla^{E\otimes \tau}$ be the Hermitian tensor product connection on $E \otimes \tau.$ Set 
$$\nabla^{E\otimes \tau}_p=\partial_t\wedge dt+ p(t)$$ to be a path of connections on $E \otimes \tau$ extended to $N,$ where $p(t)$  is a path of connections on $E\otimes \tau$ over $M.$ Let $\D^{\psi,h}_{E_i \otimes \tau, M \times [0,1]}$ be the corresponding Dirac operators at the two ends $ M \times \{i\}$.  Let $\bar{\eta}(M_i,E_i,h)= \bar{\eta}(\D^{\psi,h}_{E_i \otimes \tau, M \times [0,1]}).$
Then, we only need to compute 
\begin{align}\label{eq:welldef2}
\bar{\eta}(M_1,E_1,&h) - \bar{\eta}(M_2,E_2,h)  \nonumber \\
&- \Bigg( \int_{M_1} \Td(\Omega_{M_1}) \wedge ch(\nabla^{E_1}) \wedge f^*\mu - \int_{M_2} \Td(\Omega_{M_2}) \wedge ch(\nabla^{E_2})  \wedge f^*\mu \Bigg) \text{ mod } \Z
\end{align} where $\Omega_{M_i}$ is the respective Riemannian curvature of $M_i$ for $i=1,2.$

Let $\theta$ be the transgression form of $\Td \wedge ch$ on $N$ satisfying $$d \theta= \Td(\Omega_{M_1}) \wedge ch(\nabla^{E_1}) - \Td(\Omega_{M_2}) \wedge ch(\nabla^{E_2}).$$ The integral part of \eqref{eq:welldef2} is immediate: 

\begin{equation}
\int_{M_1} \Td(\Omega_{M_1}) \wedge ch(\nabla^{E_1}) \wedge f^*\mu - \int_{M_2} \Td(\Omega_{M_2}) \wedge ch(\nabla^{E_2})  \wedge f^*\mu = \int_{\partial N} d \theta \wedge f^*\mu \text{ mod } \Z.
\end{equation}

By the Dai-Zhang Toeplitz index formula \cite[Theorem 2.3]{dai2006index}, upon reducing modulo $\Z$, 
\begin{equation}\label{eq:welldef2a}
\bar{\eta}(M_1,E_1,h) - \bar{\eta}(M_2,E_2,h)=\int_{N} \Td(\Omega_{g_\gamma}) \wedge ch(\nabla^E_p) \wedge ch(h,d)  \text{ mod } \Z
\end{equation} where $\Omega_{g_\gamma}$ is the respective Riemannian curvature of $N$ and $ch(h,d)$ is the odd Chern character of $h$.  By Stokes theorem, the left hand side of \eqref{eq:welldef2a} is $\int_{\partial N} \theta \wedge ch(h,d).$ By the exactness condition \eqref{eq:quantization1} and Stokes theorem again,  the difference  \eqref{eq:welldef2} is zero. 
\end{proof}


\subsubsection{Well-defined under the $\R/\Z$ $K^0$-relation}

\begin{proposition} \label{k0wd2}
	The analytic pairing \eqref{eq:K0pairing} respects the $\R/\Z$ $K^0$-relation \eqref{eq:k0relation0}. 
\end{proposition}

\begin{proof}
Fix a $K_0$-cycle $(M,E,f).$ For $i=1,2,3,$ consider $\R/\Z$ $K^0$-cocycles $\cE_i=(r_i,(d,r^{-1}_i dr_i),\mu_i)$   such that $\cE_2=\cE_1 + \cE_3,$ i.e. $r_2 \simeq r_1 \oplus r_3$ and satisfying \eqref{eq:k0rel}: $$\mu_2-\mu_1 -\mu_3= Tch(r_1, r_2,r_3).$$ Let $h_i=r_i \circ f : M \to U(N_i).$ For simplicity, we denote \eqref{eq:K0pairing} by  $\bar{\eta}(\cE^M_i)$  for each $i.$ 
Assume there is a smooth path $h_t$ connecting $h_2$ and $h_1 \oplus h_3,$ both of which sit at each end of the cylinder $M \times [0,1]$ respectively. Moreover, assume  that the extension to the cylinder is compatible with all of the relevant data associated to each end, for instance there is a path $\tilde{\mu_t}=f^*\mu_t$  connecting $\tilde{\mu}_2$ at $ M \times \{0\}$ and $\tilde{\mu}_1+\tilde{\mu}_3$ at $ M \times \{1\}$ by some suitable cut-off function. Then, by the Dai-Zhang Toeplitz index theorem for manifolds with boundary \cite[Theorem 2.3]{dai2006index} and Stokes theorem, we compute
\begin{align*}
\bar{\eta}(\cE^M_2) &-\bar{\eta}(\cE^M_1) -\bar{\eta}(\cE^M_3)  \\
&= \int_{[0,1] \times M} \Td([0,1] \times M) \wedge ch([0,1] \times E) \wedge  \big(  ch(h_t;t\in [0,1]) -d\tilde{\mu}_t \big)  \text{ mod }\Z \\
&= \int_M \int^1_0 \Td(M) \wedge ch(E) \wedge \big(  ch(h_t;t\in [0,1]) -d\tilde{\mu}_t \big) \text{ mod }\Z  \\
&= \int_M \Td(M) \wedge ch(E) \wedge \big( \text{T}ch(h_1, h_2, h_3 ) -(\tilde{\mu}_2 - \tilde{\mu}_1- \tilde{\mu}_3) \big) \text{ mod }\Z =0. 
\end{align*} 
This shows that $\bar{\eta}(\cE^M_2) = \bar{\eta}(\cE^M_1) +\bar{\eta}(\cE^M_3)$ whenever $r_2 \simeq r_1 \oplus r_3.$ 
\end{proof}

\subsubsection{Well-defined under the $K$-homology relations}

\begin{lemma}\label{etawd2}
	The analytic term $\bar{\eta}\big( \D^{\psi,h}_{E \otimes \tau, M \times [0,1]} \big)$ respects the $K$-homology relations \cite[\S 11]{baum1982k}. 
\end{lemma}

\begin{proof} \vspace{-0.5em}
	The following approach is inspired by \cite{benameur2006differential}. For simplicity, we denote $\bar{\eta}\big( \D^{\psi,h}_{E \otimes \tau, M \times [0,1]} \big)$  by $\bar{\eta}(M,E,h).$ It is straightforward for the case of direct sum-disjoint union, i.e.   
	\begin{equation}
	\bar{\eta}\big((M,E_1,h) \sqcup (M,E_2,h) \big)= \bar{\eta}(M,E_1 \oplus E_2,h)= \bar{\eta}(M,E_1,h) + \bar{\eta}(M, E_2,h)
	\end{equation} where the Dirac operator splits into $\D^{\psi,h}_{E_1 \otimes \tau, M \times [0,1]}  \oplus \D^{\psi,h}_{E_2 \otimes \tau, M \times [0,1]}.$
	For bordism, let $(W,F,\varphi)$ be a $K$-chain such that  $(\partial W, F|_{\partial W},\varphi|_{\partial W}) \cong (M,E,f) \sqcup (-M',E',f').$ Then, 
	 \begin{align*}
	 \bar{\eta}(\partial W, F|_{\partial W},\varphi|_{\partial W}) 
	 = \bar{\eta}\big((M,E,f) \sqcup (-M',E',f')\big) 
	 = \bar{\eta}(M,E,f) + \bar{\eta}(-M', E',f)
	 \end{align*} where the Dirac operator is given by $\D^{\psi,h}_{E \otimes \tau, M \times [0,1]} \oplus \D^{\psi,h'}_{E' \otimes \tau', M' \times [0,1]}.$
	 
	The relation of vector bundle modification is given by  
	\begin{equation}
	(M,E,f)\sim 	(\Sigma H, \beta \otimes \rho^*E, f \circ \rho),
	\end{equation} where $H$ is  a $\spinc$ vector bundle over $M,$ $\underline{\R}$ is the trivial real line bundle, $\Sigma H=S(H \oplus \underline{\R})$ is the sphere bundle, $\rho :\Sigma H \to M$ is the projection and  $\beta$ is the Bott bundle over $\Sigma H.$ Since $M$ is an even dimensional $\spinc$ manifold, so is $\Sigma H.$ Thus, the consideration of the Dai-Zhang eta-invariant $\eta(\Sigma H, \beta \otimes \rho^*E, f \circ \rho)$ is valid. 
	Via  $r: X \to U(N),$ the composition $g=r \circ f$  defines an element in $K^1(M)$ and  $h=g \circ \rho: S(H \oplus \underline{\R}) \to U(N)$  defines an element in $K^1(S(H \oplus \underline{R})).$ Let $\tau$ be the trivial bundle where $g$ acts on and $S_M$ be the spinor bundle on $M.$ Now, we extend the tensor product bundle $S_M \otimes E \otimes \tau$ on $M$ trivially  to the cylinder $M \times [0,1],$ denoted by $S_{M\times [0,1]} \otimes F$ for $F=E \otimes \tau$.  By the Dai-Zhang construction, we obtain the associated Dirac operator  $\D^{\psi,g}_{F, M \times [0,1]}.$ Let $\widetilde{\D^{\psi,g}_{F, M \times [0,1]}}$ be its lift to $\Sigma H \times [0,1].$ This requires some explanation. Note that, there is a lift $\widetilde{S_{M\times [0,1]}}$ of $S_{M\times [0,1]} \otimes F$ to $\Sigma H \times [0,1]$ via $\rho'=\rho \times t$ where $t\in [0,1].$ Let $S_{S^{2p}}$ be the spinor bundle on the even spheres $S^{2p}.$ Denote its lift to $\Sigma H \times [0,1]$ by $\widetilde{S_{S^{2p}}}.$  Then, by \cite{atiyah1968index} there is an isomorphism of the tensor product 
	$$S_{\Sigma H \times [0,1]} \cong \widetilde{S_{M\times [0,1]}}  \hat{\otimes} \widetilde{S_{S^{2p}}}$$ where $S_{\Sigma H \times [0,1]} $ is the primitive spinor bundle associated to  $T(\Sigma H \times [0,1])$ of the $\spinc$ manifold $\Sigma H \times [0,1].$ The `full' bundle data on $\Sigma H \times [0,1]$ is now 
	\begin{equation}\label{eq:bdledata1}
	S_{\Sigma H \times [0,1]} \otimes \widetilde{\beta} \otimes (\rho')^* F.
	\end{equation}  Let  $\D_{\beta,S^{2p}}$ be the Dirac operator on $S^{2p}$ twisted by the Bott bundle $\beta,$ with $\widetilde{\D_{\beta,S^{2p}}} $ its lift to $\Sigma H \times [0,1],$ acting on \eqref{eq:bdledata1} via $\widetilde{S_{S^{2p}}}$ and $\widetilde{\beta}$ and by the identity on others.  On the other hand,  the lift $\widetilde{\D^{\psi,g}_{F, M \times [0,1]}}$ acts on \eqref{eq:bdledata1} via $\widetilde{S_{M\times [0,1]}}$ and $(\rho')^* F$ and by the identity on others.  That is, both of the lifted Dirac operators $\widetilde{\D^{\psi,g}_{F, M \times [0,1]}}$ and $\widetilde{\D_{\beta,S^{2p}}} $ act on the  bundle \eqref{eq:bdledata1}, as well as the primitive $\spinc$ Dirac operator $\D^{\psi',h}_{\widetilde{\beta} \otimes (\rho')^*F,\Sigma H \times [0,1]}.$ 
	Let $P$ be the sharp product of the two operators
	\begin{equation}
	P=\widetilde{\D^{\psi,g}_{F, M \times [0,1]}} \: \# \: \widetilde{\D_{\beta,S^{2p}}} 
	=\begin{pmatrix} \widetilde{\D^{\psi,g}_{F, M \times [0,1]}} \otimes 1 & 1 \otimes \widetilde{\D_{\beta,S^{2p}}}^- \\ 1 \otimes \widetilde{\D_{\beta,S^{2p}}}^+ & -\widetilde{\D^{\psi,g}_{F, M \times [0,1]}} \otimes 1  \end{pmatrix}.
	\end{equation}
	It is an elliptic operator on $\Sigma H\times [0,1]$ acting on  the bundle \eqref{eq:bdledata1}. Moreover, $P$ can be identified with the primitive  Dirac operator  $\D^{\psi',h}_{\widetilde{\beta} \otimes (\rho')^*F,\Sigma H \times [0,1]}$ on $\Sigma H\times [0,1]$ by the local triviality of the fibration $\Sigma H\times [0,1] \to M \times [0,1].$ One can alternatively view $P$ as the sharp product $$\widetilde{\D^{\psi,g}_{F, M \times [0,1]}} \: \# \: \mathcal{D}_{\beta}$$ 
	where $\mathcal{D}_{\beta}$ is a family of elliptic operators $\mathcal{D}$ given by the Dirac operator on $S(H_m \oplus \underline{\R})$ for $m \in M \times [0,1],$ and $(\mathcal{D}_{\beta})_m$  is identified with $\widetilde{\D_{\beta,S^{2p}}}$  by \cite[Proposition 7]{baum2018}. 
	
	All of the Dirac operators considered above are ordinary. So, we may apply the usual formula as in \cite{atiyah1976spectral3} or  \eqref{eq:sharpeta0}. The eta-invariant of the sharp product operator $P$ can be calculated by 
	$$\eta(P)=\text{Ind}\Big(\widetilde{\D_{\beta,S^{2p}}}^+  \Big) \cdot \eta\Big( \widetilde{\D^{\psi,g}_{F, M \times [0,1]}} \Big) = \eta\Big( \widetilde{\D^{\psi,g}_{F, M \times [0,1]}} \Big)$$
	since $\text{Ind}\Big(\widetilde{\D_{\beta,S^{2p}}}^+  \Big) =1$ by the Atiyah-Singer index theorem.
	This shows that 
	$$\eta\Big(\D^{\psi',h}_{\widetilde{\beta} \otimes (\rho')^*F,\Sigma H \times [0,1]} \Big) =\eta\Big( \widetilde{\D^{\psi,g}_{F, M \times [0,1]}} \Big).$$ 
	The rest of the proof involves the argument of the dimension of the kernel of the Dirac operator, which is standard. In particular, the kernel of $P$ or equivalently $\D^{\psi',h}_{\widetilde{\beta} \otimes (\rho')^*F,\Sigma H \times [0,1]}$ coincides with the kernel of $\widetilde{\D^{\psi,g}_{F, M \times [0,1]}}.$ Thus, the reduced eta-invariant is invariant under vector bundle modification.
	\end{proof}

\begin{lemma}\label{integralwd2}
	The integral term $\int_M f^*\mu \wedge ch(E)\wedge \Td(M) \text{ mod } \Z$ respects the $K$-homology relations \cite[\S 11]{baum1982k}. 
\end{lemma}
	\begin{proof}
	Fix a $\R/\Z$ $K^0$-cocycle $\cV=(g,(d,g^{-1}dg),\mu).$ Let $\cE=(M,E,f)$ be a $K_0$-cycle. For direct sum-disjoint union, it is straightforward to see that the integral of the sum splits into the sum of the integral. For bordism, consider a $K$-chain $(W,F,g)$ and by pairing $\cV$ with each term in $(\partial W,F|_{\partial W},g|_{\partial W}) \cong (M,E,f) \sqcup (-M',E',f'),$  it is immediate that 
	\begin{align*}
	\int_W (&g|_{\partial W})^*\mu \wedge ch(F|_{\partial W}) \wedge \Td(\partial W) \text{mod }\Z \\
	&=\int_M f^*\mu  \wedge ch(E)  \wedge \Td(M) \text{mod }\Z + \int_{M'} {f'}^*\mu  \wedge ch(E')  \wedge \Td(M') \text{mod }\Z.
	\end{align*}
	For vector bundle modification,  $(M,E,f) \sim (\Sigma H, \beta \otimes \rho^*E, f\circ \rho),$ we compute 
	\begin{align*}
	\int_{\Sigma H} &(f \circ \rho)^* \omega \wedge ch(\beta \otimes \rho^*F) \wedge \Td(\Sigma H) \text{ mod } \Z\\
	&= \sum_{U_\alpha} \varphi_\alpha \int_{U_\alpha \times S^{2p}} f^*(\omega|_{U\alpha}) \wedge ch(\beta) \otimes ch(E|_{U_\alpha}) \wedge \Td(U_\alpha \times S^{2p}) \text{ mod } \Z \\
	&= \sum_{U_\alpha} \varphi_\alpha \int_{U_\alpha} f^*(\omega|_{U\alpha}) \wedge ch(E|_{U_\alpha}) \wedge \Td(U_\alpha) \int_{S^{2p}} ch(\beta) \wedge \Td(S^{2p}) \text{ mod } \Z \\
	&= \sum_{U_\alpha} \varphi_\alpha \int_{U_\alpha} f^*(\omega|_{U\alpha}) \wedge ch(E|_{U_\alpha}) \wedge \Td(U_\alpha) \text{ mod } \Z\\
	&= \int_M f^*\omega \wedge ch(E )\wedge \Td(M) \text{ mod } \Z.
	\end{align*} 
	Here, $\varphi_\alpha$ is a partition of unity subordinate to an open cover $\{U_\alpha\}$ of $M$ and the second integral (over $S^{2p}$) on the second line is known to have index 1 by  the Atiyah-Singer index theorem. This completes the proof. 
\end{proof}

\begin{proposition} \label{k0wd3}
	The analytic pairing  \eqref{eq:K0pairing}  respects the $K$-homology relations \cite[\S 11]{baum1982k}.
\end{proposition}

\begin{proof}
The link between these two terms (via the exactness condition \eqref{eq:k0relation0}) does not play a role here, so the claim follows from Lemma \ref{etawd2} and Lemma \ref{integralwd2}. 
\end{proof}

\subsection{Non-degeneracy of $K^0$ pairing}
We show the non-degeneracy by an argument of Mayer-Vietoris sequence for the $K^0$ pairing.  The approach adapted here is inspired by  Savin-Sternin \cite{savin2002eta}, in which their argument works for the duality pairing on abstract cycles. In contrast, the following proof is much more delicate as explicit (co)cycles are involved.      First, we show that  \eqref{eq:K0pairing} is an isomorphism for a contractible open set $U \cong \R^n.$ Then, by the assumption that the isomorphism holds for contractible $U,V$ and intersection $U\cap V,$ it holds for $X=U\cup V.$ Lastly, we apply an induction on the size of open covers.  

To do this, we need a description of  $K_0(U) \cong K_0(\R^n)$ for positive even $n.$
Consider the short exact sequence of the induced $K_0$ groups associated to the one-point compactification of the Euclidean space $\R^n$
\begin{equation}
0 \longrightarrow  K_0(\{\infty\}) \longrightarrow K_0(S^n) \longrightarrow K_0(\R^n) \longrightarrow 0.
\end{equation}
Recall that for even $n,$ the geometric $K$-homology of even sphere $K_0(S^n)$ is 
\begin{equation}
K_0(S^n) \cong \Z\langle (\pt,\pt \times \C, i) \rangle \oplus \Z \langle (S^n, \beta, \id) \rangle.
\end{equation} Here $i : \{\infty\} \to S^n$ is the inclusion map and $\beta$ is the non-trivial Bott bundle over $S^n.$ Since $K_0(\{\infty\}) \cong \Z$ is generated by $(\pt, \pt \times \C, \id)$ and coincides with $\ker[K_0(S^n) \to K_0(\R^n)],$ we obtain $$K_0(\R^n) \cong \widetilde{K}_0(S^n)$$ where $\widetilde{K}_0(S^n)$ denotes the reduced $K$-homology of $S^n,$ generated by the non-trivial cycle. Thus,  it suffices to consider the pairing in $\widetilde{K}_0(S^n).$ 

\begin{remark}
	For $n=2,$ recall that there is a canonical line bundle  $L_0$ over $S^2 \cong \C P^1.$ Then, the Bott bundle is $\beta_0=L_0 -1$ where $1$ is the trivial line bundle, denoted by $\beta_0 \in \widetilde{K}^0(S^2).$ To see the Bott bundle over $n$-spheres for $n>2,$ we observe that by the multiplicative property of reduced $K$-theory of $S^2$
	$$\widetilde{K}^0(S^2) \times \cdots \times \widetilde{K}^0(S^2) \to \widetilde{K}^0(S^2 \wedge \cdots \wedge S^2) =\widetilde{K}^0(S^n)$$
	where $n=2r$ for $r$ times the wedge of 2-spheres. Then, the Bott bundle $\beta \in \widetilde{K}^0(S^n) $ is the $r$-th exterior tensor product of $\beta_0$ $$\beta = \beta_0 \boxtimes \cdots \boxtimes \beta_0= (L_0-1)^r.$$ Since $ch(L_0-1)=c_1(L_0),$ by the multiplicative property of the Chern character, we have $ch(\beta)=c_1(L_0)^r$  which is integral. 
\end{remark}
Consider the following exact sequence
$$\cdots \to \widetilde{K}^0(S^n) \to \widetilde{K}^0(S^n,\R) \to \widetilde{K}^0(S^n,\R/\Z) \to \widetilde{K}^1(S^n) \to \cdots$$
Recall that the odd $K$-theory $K^1(S^n)$ can be regarded as the set  of homotopy classes $[S^n,U(\infty)]$ of continuous maps from $S^n$ to the stabilised unitary group $U(\infty),$ which is by definition the $n$-th homotopy group $\pi_n(U(\infty)).$ By Bott Periodicity, $K^1(S^n) \cong \pi_n(U(\infty))$ is trivial when $n$ is even. Hence we have 
$$\cdots \to \widetilde{K}^0(S^n) \xrightarrow{ch} \widetilde{K}^0(S^n,\R) \to \widetilde{K}^0(S^n,\R/\Z) \to 0.$$ By viewing $\widetilde{K}^0(S^n,\R/\Z)$ as the cokernel of $ch,$ its generator can be represented by 
\begin{equation}
(0,0, \mu - ch(\beta)), 
\end{equation} where $\mu \in \Omega^{\text{even}}(S^n)/d\Omega$ such that $d\mu=0$ and $\beta \in \widetilde{K}^0(S^n).$ 

Now, we are ready to show that the Pontryagin duality $K^0$ pairing implemented by \eqref{eq:K0pairing} is non-degenerate.

The case of the $K^0$ pairing for $\R^n$ for positive even $n$ reduces to the $\widetilde{K}^0$ pairing for $S^n.$ In particular, it suffices to show that the map
$$\widetilde{K}^0(S^n,\R/\Z) \to \text{Hom}(\widetilde{K}_0(S^n),\R/\Z) $$ implemented by 
\begin{equation}\label{eq:k0snpairing}
\bar{\eta}\big(\D^\beta_{S^n \times [0,1]}\big) -\int_{S^n} \big(\mu-ch(\beta)\big) \wedge ch(\beta) \wedge \Td(S^n) \text{ mod } \Z
\end{equation} 
is an isomorphism, which then implies the non-degeneracy of the pairing. Since we are working on generators, the injectivity is implied and we only need to show the surjectivity, i.e. it suffices to show that the image of the pairing is not identically zero in $\R/\Z.$ Since $TS^n$ is stably trivial, the Todd form $\Td(S^n)=1.$  The integrand then consists of two parts: $$ \mu \wedge ch(\beta) \text{  and  }   ch(\beta) \wedge ch(\beta).$$ It is clear that the  integration of $ch(\beta)^2 $ over $S^n$ is zero modulo $\Z$ since $ch(\beta)$ is just the wedge product of $c_1(L)$ and is already the top degree form on $S^n.$  For $\mu \wedge ch(\beta),$  since $ch(\beta)$ is already the top degree form on $S^n,$ only the lowest term (the 0-form of $\mu$) survives in the integration. The 0-form is in general an $\R$-valued function on $S^n.$ Hence, we conclude that
\begin{equation}\label{eq:k0int1}
\int_{S^n} \mu \wedge ch(\beta) \wedge \Td(S^n) \text{ mod }\Z  \neq 0. 
\end{equation}

To consider the reduced Dai-Zhang eta-invariant $\bar{\eta}$ of the even sphere $S^n,$  we need to compute $\eta \big(\D^\beta_{S^n \times [0,1]}\big).$ Let $\D_{\beta, S^n}=\D^{\pm}_{\beta,S^n}$ be the $\Z_2$-graded Dirac operator on $S^n$ twisted by $\beta$. By the Atiyah-Singer index theorem,  
\begin{equation}
\text{Ind}(\D^+_{\beta, S^n})=\int_{S^n} ch(\beta) \wedge \Td(S^n) =1 \neq 0.
\end{equation}Hence, we cannot directly apply the method of Dai-Zhang to compute the eta-invariant, as it requires  the vanishing of $\ind(\D^+_{\beta, S^n})$ to ensure the existence of Lagrangian subspaces  in $\ker(\D_{\beta, S^n})$ for the modified boundary conditions. To circumvent this, we adopt a method suggested in \cite{dai2011eta}.   

First, we extend $\beta$ over $S^n$ trivially to the cylinder $S^n \times [0,1].$ Then, the two ends of the interval $[0,1]$ are identified into a circle, and glue the bundle over $S^n \times \{0\}$ and $S^n \times \{1\}$ using a smooth automorphism in $K^1(S^n).$ Since $K^1(S^n)$ is trivial for even $n,$ the gluing map is just the identity (infinite) matrix in $U(\infty)$ and the bundles at the two ends are identified trivially, giving a well-defined bundle $\beta' \to S^n \times S^1.$ Moreover, $S^n \times S^1$ is closed and therefore no boundary conditions 
are required. Let  $\D_{\beta', S^n\times S^1}$  be the resulting twisted Dirac operator. It can be rewritten as the sharp product
\begin{equation}\label{eq:hash1}
\D_{\beta', S^n\times S^1} = \D_{\beta, S^n} \#  \D_{S^1}= \begin{pmatrix} \D_{S^1} \otimes 1 & 1 \otimes \D^{-}_{\beta,S^n} \\ 1 \otimes \D^{+}_{\beta,S^n} & -\D_{S^1} \otimes 1 \end{pmatrix}.
\end{equation}
Then, its Atiyah-Patodi-Singer eta-invariant is $$\eta_{\text{APS}}(\D_{\beta', S^n\times S^1}) =\text{Ind}(\D^+_{\beta, S^n}) \cdot \eta_{\text{APS}}(\D_{S^1} )=1 \cdot 0=0.$$
To determine the kernel of \eqref{eq:hash1}, let 
$
\begin{pmatrix}
x_1 \otimes y_1 \\ x_2 \otimes y_2
\end{pmatrix}$ be the spinors.  Then, the calculation reduces to 
\begin{equation}
\begin{cases}
\D_{S^1} (x_1) \otimes y_1 = -x_2 \otimes \D^{-}_{\beta,S^n}(y_2)\\
\D_{S^1} (x_2) \otimes y_2 = x_1 \otimes \D^{\beta,+}_{S^n}(y_1).
\end{cases}
\end{equation}
For instance, if $(x_1,x_2) \in \ker(\D_{S^1}),$ then $(y_1,y_2) \in \ker (\D^{\beta,+}_{S^n} \oplus \D^{\beta,-}_{S^n}).$ In particular, we have 
\begin{equation}
\ker(\D^{\beta'}_{S^n\times S^1}) \cong \ker(\D_{S^1}) \dot{\cap} \big( \ker (\D^{\beta,+}_{S^n}) \oplus \ker(\D^{\beta,-}_{S^n})\big)
\end{equation} where $\dot{\cap}$ means the `intersection' of elements as in spinors, not as the  intersection of spaces.  Since the Dirac operator $\D^{+}_{\beta,S^n}$ has one dimensional kernel and zero dimensional cokernel, we conclude that 
\begin{equation}
\text{dim}\ker(\D_{\beta', S^n\times S^1}) \equiv \text{dim}\ker (\D_{S^1})=1.
\end{equation} 
Hence, for all positive even $n,$ the reduced Dai-Zhang eta invariant is 
\begin{equation}\label{eq:k0eta1}
 \bar{\eta}(\D_{\beta',S^n \times S^1})  \equiv \frac{1}{2} \text{ mod }\Z.
\end{equation}

By \eqref{eq:k0int1} and  \eqref{eq:k0eta1}, we conclude the following lemma.

\begin{lemma}\label{snpairing0}
	For even $n \in \Z_+,$ the map $\widetilde{K}^0(S^n,\R/\Z) \to \text{Hom}(\widetilde{K}_0(S^n),\R/\Z)$ implemented by \eqref{eq:k0snpairing} is an isomorphism, and thus so is the case of $U \cong \R^n.$  
\end{lemma}

Since $K_1(\R^n) \cong K_1(S^n)\cong 0$ for even $n,$ the relevant part of the Mayer-Vietoris sequence in the analytic $K^0$ pairing is 
\begin{center}
	\begin{tikzcd}[column sep = tiny]
		\arrow[r]	 & K^{0}_c(U \cap V,\R/\Z) \arrow[d] \arrow[r] & K^{0}_c(U,\R/\Z)\oplus K^{0}_c(V,\R/\Z) \arrow[d]  \arrow[r] & K^{0}(U \cup V,\R/\Z) \arrow[d] \arrow[r]  & \; \\
		\arrow[r]	 & \text{Hom}(K_{0}(U \cap V) , \R/\Z)  \arrow[r] &\text{Hom}(K_{0}(U)  , \R/\Z) \oplus \text{Hom}(K_{0}(V) , \R/\Z) \arrow[r] &\text{Hom}(K_{0}(U \cup V) , \R/\Z) \arrow[r]  & 
		\bigskip
	\end{tikzcd}  
\end{center}

\begin{lemma}\label{pairingisom0}
	Assume the isomorphism holds for contractible  open sets $U, V$ and the intersection $U \cap V.$  Then, it holds for $X=U \cup V,$ i.e. the map $K^0(X,\R/\Z) \to \text{Hom}(K_0(X),\R/\Z)$ implemented by \eqref{eq:K0pairing} is an isomorphism.
\end{lemma}

\begin{proof}
	The result follows by Lemma \ref{snpairing0}  and by the Five lemma.
\end{proof}

\begin{proof}(of  Theorem \ref{nondegenpair0}) The non-degeneracy of the pairing is implied by the isomorphism as in Lemma \ref{pairingisom0}. The last step is to induct on the size of open cover of $X.$ The base case is that we have shown $K^i(X,\R/\Z) \to \text{Hom}(K_i(X),\R/\Z)$ implemented by the analytic pairing  \eqref{eq:K0pairing}  is an isomorphism for $U,V$ and $U \cap V,$ where $\{U,V\}$ is a good cover of $X.$ 
	
Let $\{U_0, \ldots, U_{p-1} \}$ be any open cover of $X$ of size $p.$ Let $V=U_0 \cup \cdots \cup U_{p-2}.$ The induction hypothesis is the following: assume that the isomorphism holds for $V, U_{p-1}$  and the non-empty intersection $V \cap U_{p-1},$ which then holds for $V \cup U_{p-1}.$
Now, consider a good cover $\{U'_0,\ldots, U'_p\}$ of $X$ of size $p+1.$ Let $$V'=U'_0 \cup \cdots U'_{p-1}.$$ 
By the induction hypothesis, the isomorphism holds for $V'$ (since $V'$ is of size $p$ union) and $U'_p.$ The Mayer-Vietoris sequence for $V',U'_p$ and $V' \cap U'_p$ is 
$$\cdots \to K^i(V' \cap U'_p,\R/\Z) \to K^i(V',\R/\Z) \oplus K^i(U'_p,\R/\Z) \to K^i(V' \cup U'_p,\R/\Z) \to \cdots.$$
To claim  the isomorphism for the union $V' \cup U'_p$, we only need to consider the intersection. Note that $$V' \cap U'_p= (U'_0 \cap U'_p) \cup \cdots \cup (U'_{p-1} \cap U'_p).$$ It is the $p$-union of contractible sets. By the induction hypothesis, the isomorphism holds for $V' \cap U'_p.$ By the Mayer-Vietoris sequence and the Five lemma, we conclude that the isomorphism holds for $V' \cup U'_p.$  This completes the proof of the non-degeneracy of the analytic $K^0$ pairing.
\end{proof}




\bibliography{analyticPD19-final}{}
\bibliographystyle{amsplain}


\end{document}